\documentclass[11pt]{article}

\usepackage[a4paper,margin=1.15in]{geometry}
\usepackage{amsmath,amssymb,amsthm,mathtools}
\usepackage{enumitem}
\usepackage{booktabs,tabularx,array,needspace}
\usepackage{microtype}
\usepackage[noadjust]{cite}
\usepackage[hidelinks]{hyperref}

\newtheorem{theorem}{Theorem}[section]
\newtheorem{lemma}[theorem]{Lemma}
\newtheorem{corollary}[theorem]{Corollary}
\newtheorem{proposition}[theorem]{Proposition}
\newtheorem{question}{Question}[section]
\theoremstyle{remark}
\newtheorem{remark}[theorem]{Remark}
\newtheorem{definition}[theorem]{Definition}

\theoremstyle{plain}
\newtheorem{maintheorem}{Theorem}

\newcommand{\R}{\mathbb R}
\newcommand{\C}{\mathbb C}

\newcommand{\Z}{\mathbb Z}
\newcommand{\Hm}{\mathcal H}
\newcommand{\dist}{\operatorname{dist}}
\newcommand{\diam}{\operatorname{diam}}
\newcommand{\supp}{\operatorname{supp}}
\newcommand{\Crit}{\operatorname{Crit}}
\newcommand{\ind}{\operatorname{ind}}
\newcommand{\tr}{\operatorname{tr}}
\newcommand{\diver}{\operatorname{div}}
\newcommand{\Dgm}{\operatorname{Dgm}}

\newcommand{\dd}{\,d}

\newcommand{\Ep}{E_p}

\newcommand{\loc}{\mathrm{loc}}

\newcommand{\avg}{\operatorname{avg}}
\newcommand{\Sph}{\mathbb S}
\hypersetup{
 pdftitle={Monotone Sobolev functions: approximation, critical points, and level sets},
 pdfauthor={Deguang Zhong},
 pdfkeywords={Lebesgue monotonicity, Sobolev approximation, p-harmonic functions, level sets}}
\title{\bfseries Monotone Sobolev functions:\\
approximation, critical points, and level sets}
\author{Deguang Zhong\\[0.4em]
\small Institute of Applied Mathematics, Shenzhen Polytechnic University,\\
\small Shenzhen 518055, China\\
\small \texttt{huachengzhon@163.com}}
\date{}
\begin{document}
\maketitle

\begin{abstract}
We give an affirmative answer to the planar local smoothing problem in
Question~1.7 of D.~Ntalampekos and positive and negative answers to the
basic approximation and level-set parts of his higher-dimensional
Question~1.8. In every dimension $n\ge2$, each continuous Lebesgue
monotone function in $W^{1,p}$ on a bounded open set admits uniform and
strong $W^{1,p}$ approximation by monotone $C^{1,\alpha}_{\loc}$
functions, with unchanged Sobolev boundary values and no increase of
the $p$-Dirichlet energy, for $1<p<\infty$. The energy can be made
strictly smaller unless the original function is $p$-harmonic.
In the plane we obtain smooth local replacement at every isolated
$p$-harmonic critical point, with arbitrarily small $C^1$ and Sobolev
error. A point of gradient index $-m$ can be resolved into exactly
$m$ nondegenerate saddles. Together with Ntalampekos's planar theorem,
this gives smooth monotone density with fixed boundary values; at a
nonsmooth $p$-harmonic minimizer the energy increase is unavoidable
but can tend to zero. In dimensions $n\ge3$, an explicit Lipschitz
monotone function has a nonmanifold point on every level in an interval,
although it admits smooth monotone approximation. A product extension
of a planar homogeneous $7$-harmonic function also rules out a general
discrete exceptional set under the exact boundary and energy constraints.
The answer to Question~1.8 thus distinguishes the five basic
approximation properties from the stronger topological and
exceptional-set conclusions. Applications include constrained integral functionals,
strict $BV$ convergence of superlevel sets, nonlinear flux convergence,
and stability of persistence diagrams under tameness assumptions.
\end{abstract}



\section{Introduction}\label{sec1}

\subsection{Ntalampekos's questions and the approximation problem}

The principal results of this paper concern two questions posed by
D.~Ntalampekos~\cite[Questions~1.7 and~1.8]{N2020a}. We answer
the planar local smoothing problem in Question~1.7 affirmatively.
For Question~1.8, we establish the higher-dimensional analogue of
the five basic approximation properties, disprove the level-set
manifold assertion, and exhibit an obstruction to the additional
discrete-exceptional-set conclusion. These answers form the common
purpose of the approximation theorems and the counterexamples below.

The maximum and minimum principles constrain a scalar function without
prescribing an equation. Lebesgue introduced this notion of monotonicity
in his treatment of the Dirichlet problem~\cite{L1907}. Ordinary
Sobolev smoothing, although available by the Meyers--Serrin
theorem~\cite{MS1964}, need not preserve both principles. The problem
therefore concerns simultaneous control of regularity, boundary
values, energy, and the geometry of a scalar field.

\Needspace{8\baselineskip}
\begin{definition}\label{def1}
Let $M$ be a smooth manifold without boundary and without compact
connected components. A continuous function $f:M\to\R$ satisfies
the maximum principle if
\[
 \max_{\overline V}f=\max_{\partial V}f
 \qquad\text{for every nonempty open }V\Subset M.
\]
It satisfies the minimum principle if $-f$ satisfies the maximum
principle. It is monotone in the sense of Lebesgue if it satisfies
both, and strictly monotone if neither extremum is attained in the
interior of $V$ for any such $V$.
\end{definition}

Here $V\Subset M$ means that $\overline V$ is compact and contained
in $M$. Unless another space is specified, $M$ is an open subset of
Euclidean space. Strict monotonicity is equivalent to the absence
of local maxima and minima, including nonstrict local extrema:
an interior maximizer on $\overline V$ is a local maximum, and a
local maximum is an interior maximizer on a sufficiently small
closed coordinate ball. Thus a $C^1$ function with nonvanishing
gradient is strictly monotone. Monotone functions may have flat
regions. Our conventions agree with~\cite{L1907,N2020a}, including
\cite[Definition~2.17]{N2020a} for strictness. Scalar Lebesgue
monotonicity differs both from coordinate order monotonicity and
from monotonicity of a mapping defined by connected point preimages.

For $1<p<\infty$, we use the standard Sobolev spaces
$W^{1,p}(\Omega)$ and $W^{1,p}_0(\Omega)$; the latter is the closure
of $C_c^\infty(\Omega)$ in the Sobolev norm~\cite{B2011,EG2015}.
The condition $v-u\in W^{1,p}_0(\Omega)$ specifies unchanged
Sobolev boundary values without any regularity assumption on
$\partial\Omega$. On a Lipschitz domain it is equivalent to
equality of traces. We write
\[
 \Ep(v;D)=\int_D|\nabla v|^p\dd x,
 \qquad \Ep(v)=\Ep(v;\Omega)
\]
when the underlying domain is understood. All these energies and
Sobolev norms are Euclidean. A subscript $\loc$ denotes local
regularity, with no global bound on the corresponding local norms.

Manfredi~\cite{M1994} developed weak monotonicity through Sobolev
truncation principles. Haj\l asz and Mal\'y~\cite{HM2002} constructed
approximations by locally weakly monotone functions while controlling
nonlinear expressions in the gradient. Their starting functions
need not be monotone, and the approximants need not be continuous
or smooth. Here continuity and scalar monotonicity are retained
throughout, and the objective is improved differentiability in a
fixed boundary-value class.

Ntalampekos~\cite[Theorem~1.4]{N2020a} proved that almost every level
of a continuous monotone planar Sobolev function is an embedded
one-dimensional topological submanifold with locally finite length.
His Theorem~1.6 gives monotone $C^{1,\alpha}_{\loc}$ approximation
with uniform and strong Sobolev convergence, fixed boundary values,
and a nonincreasing energy bound. The approximants are $p$-harmonic
on a set of measure arbitrarily close to full and smooth off a
discrete set; when $p=2$, they are smooth everywhere.
We retain the numbering of the source questions. The first concerns
local smoothing in the plane; the second asks for higher-dimensional
analogues of the two cited planar theorems.

\setcounter{question}{6}
\begin{question}\label{que1}
For a planar $p$-harmonic function with $1<p<\infty$ and $p\ne2$,
is smooth monotone approximation possible near a critical point,
simultaneously in the uniform and $W^{1,p}$ topologies?
\end{question}

This is the planar problem in~\cite[Question~1.7]{N2020a}.
No nonincreasing-energy condition is imposed there. Our local theorem
also keeps the function unchanged outside any prescribed neighborhood
and controls the $C^1$ error. The case $p=2$, included in our theorem,
already has smooth harmonic input.

\begin{question}\label{que2}
Do the level-set and approximation conclusions of
\cite[Theorems~1.4 and~1.6]{N2020a} admit higher-dimensional analogues?
\end{question}

Question~\ref{que2} restates~\cite[Question~1.8]{N2020a}, which
places no separate higher-dimensional restriction on $p$. The cited
planar level-set theorem assumes $1\le p\le\infty$, whereas the
approximation theorem assumes $1<p<\infty$. Our Lipschitz level-set
example covers every Sobolev exponent in the former range.
Theorem~\ref{thm1} gives
the affirmative answer for approximation properties (A)--(E),
whereas Theorems~\ref{thm5} and~\ref{thm6} give the negative
answers for the level-set assertion and the general discrete
exceptional-set refinement, respectively. Theorem~\ref{thm2}
answers Question~\ref{que1}; Theorems~\ref{thm3}
and~\ref{thm4} supply its saddle-resolution refinement and
global planar consequence.

\subsection{Approximation and critical-point replacement}

\begin{maintheorem}\label{thm1}
Let $n\ge2$, let $\Omega\subset\R^n$ be bounded and open, and
let $1<p<\infty$. If $u\in C(\Omega)\cap W^{1,p}(\Omega)$
is monotone, there are $\alpha=\alpha(n,p)\in(0,1)$ and monotone
functions $u_j\in W^{1,p}(\Omega)$ such that
\begin{enumerate}[label=\textup{(\Alph*)},leftmargin=2.8em]
\item $u_j\in C^{1,\alpha}_{\loc}(\Omega)$ for every $j$;
\item $\sup_\Omega|u_j-u|\longrightarrow0$;
\item $u_j\longrightarrow u$ strongly in $W^{1,p}(\Omega)$;
\item $\Ep(u_j;\Omega)\le\Ep(u;\Omega)$ for every $j$;
\item $u_j-u\in W^{1,p}_0(\Omega)$ for every $j$.
\end{enumerate}
If $u$ is not $p$-harmonic, the inequalities in
\textup{(D)} can all be made strict. No boundedness of $u$,
continuity on $\overline\Omega$, or boundary regularity is assumed.
\end{maintheorem}

The labels (A)--(E) agree with~\cite[Theorem~1.6]{N2020a};
its gradient-norm inequality in (D) is equivalent to our energy bound.
Thus Theorem~\ref{thm1} answers affirmatively the basic
approximation part of Ntalampekos's Question~1.8 in every
dimension $n\ge3$.
The exponent in (A) is uniform in $j$; the local regularity norms
need not be. The strict assertion follows from the positive
separation of the obstacles and is proved in
Corollary~\ref{cor1}. $p$-harmonicity and its energy
characterization are recalled in Section~\ref{sec2}.

The proof of Theorem~\ref{thm1} has three ingredients.
A conformal enlargement due to M\"uller and
Nardmann~\cite{MN2015} provides a complete auxiliary metric with
bounded geometry in which the given function is uniformly
continuous. Sup--inf regularization, introduced by Lasry and
Lions~\cite{LL1986}, studied in Hilbert spaces by Attouch and
Az\'e~\cite{AA1993}, and extended to Riemannian manifolds by
Azagra and Ferrera~\cite{AF2015}, then yields locally $C^{1,1}$
one-sided approximations. The semiconvex calculus is standard
\cite{CS2004}; preservation of the maximum principle by the
second envelope, including maximal plateaus, is proved here.
Finally, a double-obstacle minimizer between these approximations
is monotone by truncation. Lieberman's regularity theorem
\cite{L1991} gives its $C^{1,\alpha}_{\loc}$ regularity.
Related obstacle regularity was developed by Choe and
Lewis~\cite{CL1991}; Farnana~\cite[Theorem~3.9]{F2009} treats continuity for
continuous double obstacles. The latter continuity result alone
does not provide the differentiable approximants needed here.

For a $C^1$ difference $h$ with bounded values and gradient on $V$,
put
\begin{equation*}
 \|h\|_{C^1(V)}=\sup_V|h|+\sup_V|\nabla h|.
\end{equation*}
When this notation is applied below, it is the difference that is
bounded; no global bound on the original function or its gradient
is required. Compactly supported differences automatically have
finite norm.

The regularity theory of the $p$-Laplace equation was developed by
Uhlenbeck~\cite{U1977}, Evans~\cite{E1982},
DiBenedetto~\cite{D1983}, Lewis~\cite{L1983}, and
Tolksdorf~\cite{T1984}; see also the nonlinear potential theory
of Heinonen, Kilpel\"ainen, and Martio~\cite{HKM1993}.
In the plane, Bojarski and Iwaniec~\cite{BI1987} connected the
complex gradient to quasiregular mappings. The isolation and
regularity of critical points were studied by
Manfredi~\cite{M1988}, Iwaniec and Manfredi~\cite{IM1989}, and
Aronsson~\cite{A1986,A1989}. The quasiregular framework is
developed in~\cite{AIM2009}. We show that regularized Dirichlet
minimizers have complex gradients with a common distortion bound.
The compactness theorem of Hinkkanen and Martin~\cite{HM2020}
then upgrades variational convergence to local $C^1$ convergence.
This is precisely the control needed for joining a replacement
across an annulus on which the original gradient is nonzero.

\begin{maintheorem}\label{thm2}
Let $V\subset\R^2$ be a domain, let $1<p<\infty$, and let $u$
be $p$-harmonic in $V$ with $\Crit(u)=\{x_0\}$. For every open
neighborhood $O$ of $x_0$ in $V$ and every $\eta>0$ there is
$\widetilde u\in C^\infty(V)$ such that
\begin{enumerate}[label=\textup{(\roman*)}]
\item $\widetilde u$ has no local maximum or minimum in $V$;
\item $\widetilde u-u$ has compact support in $O$;
\item
\begin{equation}\label{eq1}
 \|\widetilde u-u\|_{C^1(V)}
 +\|\widetilde u-u\|_{W^{1,p}(V)}<\eta.
\end{equation}
\end{enumerate}
In particular, $\widetilde u$ is strictly monotone and
$\widetilde u-u\in W^{1,p}_0(V)$.
\end{maintheorem}

Theorem~\ref{thm2} gives an affirmative answer to the planar
local smoothing problem in~\cite[Question~1.7]{N2020a}, with the
additional $C^1$ control in \eqref{eq1} and a change supported
in any prescribed neighborhood. It does not require the replacement to remain
$p$-harmonic where it is changed. The next theorem gives finer
control. We use the usual real-gradient index convention, so a
nondegenerate planar saddle has index $-1$.

\begin{maintheorem}\label{thm3}
Under the hypotheses of Theorem~\ref{thm2}, suppose
$\ind_{x_0}(\nabla u)=-m$. Then $m$ is a positive integer,
and $\widetilde u$ can be chosen to have exactly $m$ critical
points in $V$, all in $O$ and all nondegenerate saddles, with
pairwise distinct critical values. All the support and
approximation conclusions of Theorem~\ref{thm2} remain valid.
\end{maintheorem}

The saddle count uses the degree of the real gradient, the planar
elliptic index theory of Alessandrini and
Magnanini~\cite{AM1992,AM1994}, Sard's theorem~\cite{S1942},
and the Morse lemma~\cite{M1963,M1965}.

The index and Morse notions used here are reviewed in
Section~\ref{sec6}. Applying the local theorem at the
discrete exceptional points of Ntalampekos's approximants gives
the global planar statement.

\Needspace{13\baselineskip}
\begin{maintheorem}\label{thm4}
Let $\Omega\subset\R^2$ be bounded and open, let $1<p<\infty$,
and let $u\in C(\Omega)\cap W^{1,p}(\Omega)$ be monotone.
For every $\eta,\delta>0$ there is a monotone
$w\in C^\infty(\Omega)\cap W^{1,p}(\Omega)$ such that
\begin{equation}\label{eq2}
 \|w-u\|_{L^\infty(\Omega)}+\|w-u\|_{W^{1,p}(\Omega)}<\eta,
 \qquad w-u\in W^{1,p}_0(\Omega),
\end{equation}
and
\begin{equation}\label{eq3}
 \Ep(w;\Omega)\le\Ep(u;\Omega)+\eta.
\end{equation}
There is an open set $G\subset\Omega$ on which $w$ is $p$-harmonic
and for which $|\Omega\setminus G|<\delta$.
If $u$ is not $p$-harmonic, one can require
$\Ep(w;\Omega)<\Ep(u;\Omega)$.
\end{maintheorem}

Theorem~\ref{thm4} is the global smooth-density consequence
of the affirmative answer to Ntalampekos's Question~1.7.
Theorem~\ref{thm1} and the planar smoothing theorem serve
different energy requirements. A $p$-harmonic function uniquely
minimizes the strictly convex Dirichlet integral in its boundary
class. Hence every nontrivial same-boundary modification has
strictly greater energy. At a nonsmooth minimizer the error term
in \eqref{eq3} is therefore necessary. At a nonminimizer,
Theorem~\ref{thm1} first produces an energy gap, and the
subsequent smooth approximation can be kept within that gap.

For mappings, Youngs~\cite{Y1948} studied homeomorphic
approximation of topologically monotone maps. Iwaniec, Kovalev,
and Onninen~\cite{IKO2011,IKO2012} proved planar diffeomorphic
Sobolev approximation; Hencl and Pratelli~\cite{HP2018}
established the $W^{1,1}$ case, and Iwaniec and
Onninen~\cite{IO2016} treated monotone Sobolev mappings of
planar domains and surfaces. Campbell, D'Onofrio, and
V\'itek~\cite{CDV2026} obtained diffeomorphic approximation
of locally finite piecewise affine homeomorphisms in dimensions
three and four, with Sobolev control of both maps and inverses.
The coordinatewise consequence proved here preserves scalar
monotonicity and boundary values; it does not preserve injectivity.

\subsection{Higher-dimensional obstructions and the scope of the results}

\begin{maintheorem}\label{thm5}
For every $n\ge3$ there is a Lipschitz monotone function
$U:(-2,2)^n\to\R$ such that, for every $t\in(-1,1)$,
$U^{-1}(t)$ has no topological manifold neighborhood at
\[
 P_t=(0,t,t,0,\ldots,0).
\]
Nevertheless, there are smooth monotone $U_N$ on the cube such that
\[
 U_N\longrightarrow U\quad\text{uniformly and strongly in }
 W^{1,p}((-2,2)^n)\quad(1\le p<\infty).
\]
Almost every level
of $U$ has finite $\Hm^{n-1}$ measure.
\end{maintheorem}

Theorem~\ref{thm5} answers negatively the level-set part of
Ntalampekos's Question~1.8 for every $n\ge3$, even for Lipschitz
functions and even when smooth monotone approximation is available.

A topological $m$-manifold is a Hausdorff, second countable space
locally homeomorphic to $\R^m$; an embedded submanifold also has
ambient charts straightening it~\cite{L2011}. Theorem~\ref{thm5}
excludes even an intrinsic manifold chart, including a chart with
boundary. Arbitrarily small loops near $P_t$ remain noncontractible
in the entire level set. The displayed smooth sequence is explicit;
it is not asserted to preserve the boundary values or the one-sided
energy bound. For a fixed $1<p<\infty$, Theorem~\ref{thm1}
provides a separate $C^{1,\alpha}_{\loc}$ sequence with those
variational properties.

We normalize integer-dimensional Hausdorff measure so that it
agrees with Euclidean Lebesgue measure on an affine subspace:
\[
 \Hm^m(E)=\lim_{\delta\downarrow0}\frac{\omega_m}{2^m}
 \inf\left\{\sum_i(\diam E_i)^m:
 E\subset\bigcup_iE_i,\ \diam E_i<\delta\right\},
\]
where $\omega_m$ is the volume of the unit ball in $\R^m$;
see~\cite{EG2015,F1969}. Finite level measure is fully consistent
with the failure of the manifold property.

Level-set topology must be distinguished from level-set measure.
Federer's coarea theory~\cite{F1959,F1969} and its Sobolev
formulation by Mal\'y, Swanson, and Ziemer~\cite{MSZ2003}
control the latter. Alberti, Bianchini, and
Crippa~\cite{ABC2013} studied Lipschitz level sets and Sard-type
properties; Bourgain, Korobkov, and Kristensen~\cite{BKK2013}
obtained results under Sobolev and bounded-variation assumptions.
Related extensions of Sard's theorem are due to de
Pascale~\cite{dP2001}, Bojarski, Haj\l asz, and
Strzelecki~\cite{BHS2005}, and Figalli~\cite{F2008}.
These regularity assumptions and conclusions do not turn every
finite-measure level of a higher-dimensional Lipschitz monotone
function into a hypersurface. Theorem~\ref{thm5} gives an
explicit separation, and Section~\ref{sec12} shows that it
persists even alongside strict $BV$ convergence of superlevel sets.

On metric surfaces, Rajala's uniformization theorem~\cite{R2017}
uses scalar minimizers and their levels. Ntalampekos's metric
level-set theorem~\cite[Theorem~1.5]{N2020a} separates the
topological argument from Euclidean differentiation, and his
work on Sierpi\'nski carpets~\cite{N2020b} provides another
potential-theoretic motivation. Esmayli and Haj\l asz~\cite{EH2021}
treated coarea inequalities in metric spaces; Esmayli, Ikonen,
and Rajala~\cite{EIR2023} established such inequalities for
monotone Sobolev functions on metric surfaces. Meier and
Ntalampekos~\cite{MN2024} proved a Sobolev coarea inequality
and applied it to Lipschitz-volume rigidity. Meier, Vikman,
and Wenger~\cite{MVW2026} proved monotone Sobolev extension
results into metric surfaces with applications to uniformization.
Our auxiliary metric is only a regularization device on a
Euclidean domain; transferring the conclusions to metric surfaces
would require separate control of their Sobolev energies.

\begin{maintheorem}\label{thm6}
For every $n\ge3$ there is a monotone $7$-harmonic function
\[
 H\in C^{1,1/2}_{\loc}((-1,1)^n)\cap W^{1,7}((-1,1)^n)
\]
such that, for every $v\in W^{1,7}((-1,1)^n)$, the two conditions
\[
 v-H\in W^{1,7}_0((-1,1)^n),\qquad
 \int_{(-1,1)^n}|\nabla v|^7\dd x
 \le\int_{(-1,1)^n}|\nabla H|^7\dd x
\]
imply $v=H$ almost everywhere. The function $H$ is not $C^2$
at any point of $\{(0,0)\}\times(-1,1)^{n-2}$.
Consequently no such $v$ has a representative smooth near every
point outside a discrete subset of $(-1,1)^n$.
\end{maintheorem}

Our homogeneous
$7$-harmonic example belongs to the classical separated-variable
family of Kr\'ol and Aronsson~\cite{A1986,A1989,K1973}; see
also Akman, Lewis, and Vogel~\cite{ALV2019}. The explicit
calculation identifies its exponent and singularity. The
higher-dimensional obstruction uses product extension and strict
energy convexity, rather than a new existence theorem for
homogeneous solutions.

A discrete subset has each of its points isolated in its relative
topology. Theorem~\ref{thm6} gives a further negative answer
within Ntalampekos's Question~1.8: the discrete-exceptional-set
conclusion of his planar theorem cannot extend to all dimensions
and exponents while both exact variational constraints are retained.

\subsection{Answers to Questions 1.7 and 1.8}

The correspondence between the source questions and our results
is summarized below. The distinction between the basic properties
(A)--(E) and the additional conclusions is essential to the answer
to~\cite[Question~1.8]{N2020a}.

\begin{center}
\small
\renewcommand{\arraystretch}{1.15}
\begin{tabularx}{\linewidth}{@{}>{\raggedright\arraybackslash}p{0.27\linewidth}>{\raggedright\arraybackslash}p{0.16\linewidth}>{\raggedright\arraybackslash}X@{}}
\toprule
Question in~\cite{N2020a} & Result & Answer and scope\\
\midrule
1.7: planar local smoothing
& Theorem~\ref{thm2}
& Affirmative, with $C^1$ control; Theorems~\ref{thm3}
and~\ref{thm4} give saddle resolution and global smooth density.\\
1.8: approximation properties (A)--(E)
& Theorem~\ref{thm1}
& Affirmative for every $n\ge3$ and $1<p<\infty$,
with fixed boundary values and nonincreasing energy.\\
1.8: almost-everywhere manifold levels
& Theorem~\ref{thm5}
& Negative for every $n\ge3$; a Lipschitz example has
a nonmanifold point on every level in an interval.\\
1.8: discrete smoothness exception
& Theorem~\ref{thm6}
& Negative in general under (D) and (E), as shown
by a $7$-harmonic example in every $n\ge3$.\\
\bottomrule
\end{tabularx}
\end{center}

The higher-dimensional $C^{1,\alpha}_{\loc}$ theorem does not
assert that its approximants are $p$-harmonic on an open set
of measure arbitrarily close to full. Nor do the counterexamples
settle general higher-dimensional $C^\infty$ approximation when
only convergence of energies, rather than the one-sided bound,
is required. The answers to Question~1.8 stated above refer to
the precise assertions in the table. In the planar smooth result,
the arbitrarily small energy increase at a nonsmooth
$p$-harmonic minimizer is necessary; see
Proposition~\ref{prop1}.

Our applications combine the approximation theorems with established
continuity arguments for integral functionals~\cite{D2008},
coarea and perimeter theory~\cite{AFP2000}, nonlinear error
variables~\cite{BL1993,DK2008}, and persistence
stability~\cite{CEH2007}. Their common feature is a recovery
sequence satisfying the monotonicity and Sobolev boundary
constraints, with the appropriate regularity in each dimension.

\subsection{Organization of the proof}

Section~\ref{sec2} supplies the shared maximum-principle,
gluing, and energy-rigidity facts. Sections~\ref{sec3}--\ref{sec5}
prove Theorem~\ref{thm1}. Sections~\ref{sec6}--\ref{sec8}
establish the planar local and global smooth theorems.
Section~\ref{sec9} constructs the nonmanifold levels, and
Section~\ref{sec10} treats singular minimizers and their products.
The final four sections develop the variational, coarea, flux,
and persistence consequences. The planar smooth theorem uses
Ntalampekos's planar construction in addition to the local
replacement proved here; it is not inferred from arbitrary-dimensional
$C^{1,\alpha}$ density alone.

\section{Maximum principles and energy minimizers}\label{sec2}

The maximum and minimum principles will be used both for the
obstacle construction and for joining local smooth replacements.
The following characterization allows nonstrict maximal plateaus.

\begin{lemma}\label{lem1}
Let $M$ be a manifold as in Definition~\ref{def1}, in particular
without compact connected components. For $f\in C(M)$, the maximum principle
is equivalent to the following condition: for every $t\in\R$, no
connected component of $\{f>t\}$ is relatively compact in $M$.
\end{lemma}

\begin{proof}
If a component $C$ of $\{f>t\}$ is relatively compact, its boundary
in $M$ is nonempty: otherwise $C$ would be a compact connected
component of $M$. Continuity and local connectedness imply
$f=t$ on $\partial C$. Applying the maximum
principle to $C$ gives a contradiction. Conversely, if the maximum
principle fails on $V\Subset M$, choose
\[
 \max_{\partial V}f<t<\max_{\overline V}f.
\]
The component of $\{f>t\}$ containing an interior point above $t$
cannot cross $\partial V$. It is therefore contained in $V$ and has
compact closure in $M$.
\end{proof}

The next gluing statement is a form of
\cite[Lemmas~2.15 and~2.19]{N2020a}. Its hypotheses distinguish
the weak maximum--minimum condition of the original function
from the absence of all local extrema in the replacement region.

\begin{lemma}\label{lem2}
Let $f\in C(V)$ be monotone, let $K\Subset A\subset V$ with $K$
compact and $A$ open, and let $w\in C(V)$ agree with $f$ on
$V\setminus K$.  If $w$ has no local extrema in $A$, then $w$ is
monotone in $V$.  Moreover, a locally uniform limit of monotone
functions is monotone.
\end{lemma}

\begin{proof}
Suppose that for some open $U\Subset V$ the maximum $m_0$ of $w$ on
$\overline U$ is larger than its maximum on $\partial U$.  Choose an
open set $B$ with $K\subset B\Subset A$.  No point of $U\cap A$
can attain $m_0$.  Hence a point attaining $m_0$ lies in $U\setminus A$,
and therefore in the open set $W=U\setminus\overline B$.
On $W$ and on its boundary, $w=f$, since these sets are disjoint
from $K$.  Applying monotonicity of $f$ to $W$ gives a point
$y\in\partial W$ with $w(y)=m_0$.  Such a point cannot lie in
$\partial U$, so it lies in $U\cap\partial B\subset U\cap A$.
It is then a local maximum of $w$, a contradiction.  The minimum is
handled by replacing $w$ and $f$ with their negatives.

For the second assertion, uniform convergence on $\overline U$
passes both extremum identities in Definition~\ref{def1} to the limit, for each fixed
$U\Subset V$.
\end{proof}

\subsection{The variational equation and its rigidity}

Let $V\subset\R^n$ be open. We use the following weak formulation.

\begin{definition}\label{def2}
Let $1<p<\infty$.  A continuous function
$u\in W^{1,p}_{\loc}(V)$ is \emph{$p$-harmonic} if
\begin{equation}\label{eq4}
 \int_V |\nabla u|^{p-2}\nabla u\cdot\nabla\varphi\dd x=0
 \quad\text{for every }\varphi\in C_c^\infty(V).
\end{equation}
At $\nabla u=0$ the vector $|\nabla u|^{p-2}\nabla u$ is defined to
be zero. We write $\Delta_pu=\diver(|\nabla u|^{p-2}\nabla u)$.
Throughout, the term $p$-harmonic refers to this weak equation with
the continuous representative.
\end{definition}

The weak formulation, the equivalent local minimizing property of
$\Ep$, and the comparison and strong maximum principles are standard;
see \cite{HKM1993}.  The regularity results cited in the introduction
give $u\in C^{1,\alpha}_{\loc}$ for some $\alpha>0$.
Consequently
\begin{equation*}
 \Crit(u)=\{x\in V:\nabla u(x)=0\}
\end{equation*}
is well defined.  The function is $C^\infty$ on
$V\setminus\Crit(u)$, by nondegenerate elliptic regularity
\cite{GT2001,T1984}.

\begin{proposition}\label{prop1}
Let $\Omega\subset\R^n$ be bounded and open, let $1<p<\infty$, and
let $u\in C(\Omega)\cap W^{1,p}(\Omega)$ be $p$-harmonic.  If
$w-u\in W^{1,p}_0(\Omega)$, then
\begin{equation}\label{eq5}
 \Ep(w;\Omega)\ge\Ep(u;\Omega),
\end{equation}
with equality if and only if $w=u$ almost everywhere.  Consequently,
if the continuous representative of $u$ is not smooth, every smooth
function with the same Sobolev boundary values has strictly greater
energy.
\end{proposition}

\begin{proof}
Strict convexity of $\xi\mapsto|\xi|^p$ gives
\[
 |\nabla w|^p\ge |\nabla u|^p
   +p|\nabla u|^{p-2}\nabla u\cdot(\nabla w-\nabla u),
\]
with equality precisely when the two gradients agree.  The integral
of the last term is zero.  Indeed the weak equation extends from
$C_c^\infty(\Omega)$ to $W^{1,p}_0(\Omega)$ by H\"older's
inequality and density, and it can be tested with $w-u$.
This proves \eqref{eq5}.  Equality forces $\nabla(w-u)=0$ almost
everywhere, and Poincar\'e's inequality on $W^{1,p}_0(\Omega)$
forces $w-u=0$.  Continuous representatives that agree almost
everywhere agree everywhere.
\end{proof}

\begin{remark}\label{rem1}
Proposition~\ref{prop1} applies even when $u$ is smooth.  Any genuine
same-boundary modification of a $p$-harmonic minimizer costs positive
energy.  Theorems~\ref{thm2} and~\ref{thm3} make this cost arbitrarily
small.  In Theorem~\ref{thm4}, an already smooth $p$-harmonic function
can of course be left unchanged.  The nonincreasing-energy version
is obstructed precisely when one asks to change a minimizer while
retaining its boundary class.
\end{remark}

\section{An auxiliary metric adapted to the function}\label{sec3}

The Sobolev norms remain Euclidean. The purpose of the auxiliary
metric is only to make a possibly unbounded continuous function
uniformly continuous at a common metric scale.

We use standard Riemannian notation: $d_g$ is geodesic distance,
$\exp_x$ is the exponential map, and tangent and cotangent vectors are
identified by $g$. Completeness and the Hopf--Rinow theorem imply
existence of minimizing geodesics and compactness of closed bounded
sets within a connected component; see~\cite{dC1992}. A smooth metric has
bounded geometry when its injectivity radius has a positive uniform
lower bound and its curvature tensor and all covariant derivatives
are uniformly bounded~\cite{MN2015}. The injectivity radius at a
point measures the radius on which its exponential map is a
diffeomorphism. A ball is strongly convex if any two of its points
are joined by a unique minimizing geodesic lying in the ball;
the convexity radius measures the range of such balls~\cite{dC1992}.
The metric constructed below has bounded curvature and positive uniform
injectivity and convexity radii. Consequently there is $r>0$ such that
the relevant balls of radius $8r$ are strongly convex, and the function
\[
 c(x,y)=\tfrac12d_g(x,y)^2
\]
is smooth near the diagonal. By taking $r$ smaller, we have
\begin{equation}\label{eq6}
 \tfrac12 g_y\le \nabla_y^2c(x,y)\le2g_y
 \qquad(d_g(x,y)<8r).
\end{equation}
The corresponding mixed derivatives of $c$ are uniformly bounded on
this neighborhood. These are the usual local squared-distance
estimates from comparison geometry; see~\cite{AF2015,dC1992}.
Only their existence, and not a sharp constant, is needed.

\begin{lemma}\label{lem3}
Let $\Omega\subset\R^n$ be bounded and open and let $f\in C(\Omega)$.
There is a complete smooth Riemannian metric $g$ on $\Omega$ with
bounded geometry, positive uniform convexity radius, and the following
property: $f$ is uniformly continuous for $d_g$ within the connected
components, with a modulus independent of the component.
\end{lemma}

\begin{proof}
Put $d(x)=\dist(x,\partial\Omega)$. By continuity, we may choose a
continuous positive function $b$ on $\Omega$ such that
\begin{equation}\label{eq7}
 b(x)\le\min\{1,d(x)/4\},\qquad
 |f(y)-f(x)|\le\min\{1,d(x)\}\quad\text{if }|y-x|<b(x).
\end{equation}
For completeness, the allowable radii have a positive lower bound
on a neighborhood of each point: first make the oscillation of $f$
smaller than half of the positive local lower bound for
$\min\{1,d\}$, and then shrink the neighborhood. A partition of unity
gives a continuous positive minorant of these allowable radii.

Define
\[
 \tau(x)=\inf_{y\in\Omega}\{b(y)+|x-y|\}.
\]
Then $0<\tau\le b$ and $\tau$ is $1$-Lipschitz. Positivity follows
by splitting the infimum into a small neighborhood of $x$, on which
$b$ is bounded below, and its complement. Choose a smooth positive
function $a\ge\tau^{-1}$ and the preliminary metric
$h=a^2\sum_i dx_i^2$.

Apply M\"uller and Nardmann~\cite[Theorem~1.4]{MN2015} with
initial metric $g_0=h$, a smooth compact exhaustion, positive constant
curvature bounds, a constant radius bound $\iota>0$, and prescribed
lower bound $u_0\equiv1$ for the logarithmic conformal factor.
The theorem gives $g=e^{2\sigma}h$ with $\sigma>1$ and $g$ complete;
\cite[Fact~1.3]{MN2015} converts its bounds outside the exhaustion
sets into bounded geometry on all of $\Omega$.
By~\cite[Remark~1.7]{MN2015}, the same theorem holds with convexity
radius in place of injectivity radius; thus we may also require
$\operatorname{conv}_g\ge\iota$.
In particular $g\ge h$, and hence
\begin{equation*}
 |\xi|_g\ge |\xi|/\tau(x),\qquad |\xi|_g\ge|\xi|.
\end{equation*}

If $\gamma$ is parametrized by $g$-arc length and starts at $x$, the
Lipschitz property of $\tau$ gives, almost everywhere,
\[
 |\gamma'|\le\tau(\gamma),\qquad
 \tau(\gamma(s))\le e^s\tau(x),\qquad
 |\gamma(s)-x|\le(e^s-1)\tau(x).
\]
Thus $d_g(x,y)<\log2$ implies $|y-x|<b(x)$ and the oscillation bound
in~\eqref{eq7} applies. Given $\eta>0$, points with $d(x)<\eta/2$
therefore satisfy $|f(y)-f(x)|<\eta$ at this fixed small metric scale.
For the remaining points, take also $d_g(x,y)<\eta/4$. Both points
then lie in the compact subset $\{d\ge\eta/4\}\subset\Omega$.
Uniform Euclidean continuity on that set, together with
$|x-y|\le d_g(x,y)$, supplies a smaller scale at which the same
oscillation bound holds. This proves the assertion. All constants
were chosen for the whole of $\Omega$.
\end{proof}

\begin{remark}\label{rem2}
The maximum principle is topological and is unchanged by the choice
of $g$. Local $C^{1,1}$ regularity for a smooth Riemannian metric is
also the usual local $C^{1,1}$ regularity in Euclidean coordinates.
There is no claim that the resulting obstacles have bounded global
Euclidean derivatives.
\end{remark}

\section{Regularization preserving one maximum principle}\label{sec4}

Throughout this section $M$ has no compact connected components and
is a complete finite-dimensional
Riemannian manifold with the uniform local geometry just described.
All envelopes are taken in the connected component of their argument.
For a function $f$ define
\begin{equation*}
 S_\lambda f(x)=\sup_{y\in M}
       \left(f(y)-\frac{d_g(x,y)^2}{2\lambda}\right),\qquad
 I_\mu f(x)=\inf_{y\in M}
       \left(f(y)+\frac{d_g(x,y)^2}{2\mu}\right).
\end{equation*}
These are the sup- and inf-convolutions associated with squared
distance~\cite{AF2015,LL1986}.

A function $F$ is uniformly locally $C$-semiconvex if there is a
fixed radius on whose geodesically convex balls
$(F\circ\gamma)''\ge-C$ in distributions for every unit-speed
geodesic $\gamma$. Its subdifferential $\partial F(y)\subset T_yM$
is the set of vectors $\zeta$ such that
\[
 F(\exp_y v)\ge F(y)+\langle\zeta,v\rangle+o(|v|)
 \quad(v\to0).
\]
For finite locally semiconvex functions it is nonempty, compact and
convex, and is upper semicontinuous in local tangent-bundle
coordinates. If $F$ is $\Lambda$-Lipschitz, all its subgradients have norm
at most $\Lambda$. These facts reduce in smooth charts to the corresponding
facts for convex functions; see~\cite{AF2015,CS2004}.

\begin{lemma}\label{lem4}
Let $f:M\to\R$ be uniformly continuous, with one modulus on all
components. For all sufficiently small $\lambda>0$, $F=S_\lambda f$
is finite, Lipschitz and uniformly locally semiconvex, with constants
which may depend on $\lambda$. Moreover,
\[
 0\le S_\lambda f-f\le e_\lambda,\qquad e_\lambda\longrightarrow0.
\]
If $f$ satisfies the maximum principle, so does $S_\lambda f$.
\end{lemma}

\begin{proof}
Let $\omega(s)$ be the supremum of $|f(x)-f(y)|$ over pairs in the
same component with $d_g(x,y)\le s$. Subdivision of minimizing
geodesics, using uniform continuity at one fixed scale, shows that
\begin{equation}\label{eq8}
 \omega(s)\le A s+B\quad(s\ge0),\qquad \omega(s)\longrightarrow0
 \quad(s\downarrow0),
\end{equation}
for constants $A,B$. Quadratic penalization is therefore coercive,
so the supremum defining $S_\lambda f(x)$ is attained. An optimizer
$y$ satisfies
\[
 \frac{d_g(x,y)^2}{2\lambda}\le f(y)-f(x)
 \le A d_g(x,y)+B.
\]
Hence all optimizers lie at distance at most
$R_\lambda=A\lambda+\sqrt{A^2\lambda^2+2B\lambda}$ from $x$,
and $R_\lambda\to0$. The bound
\[
 0\le S_\lambda f(x)-f(x)
 \le\sup_{s\ge0}\left(\omega(s)-\frac{s^2}{2\lambda}\right)
 =:e_\lambda
\]
and~\eqref{eq8} give $e_\lambda\to0$ by first restricting to small
$s$ and then using the quadratic penalty on the complement.

Fix a uniform geometric radius $r>0$ and take $R_\lambda<r$.
On $B_g(x_0,r)$ the supremum may be restricted to
$y\in\overline{B_g(x_0,2r)}$. For these $x,y$, the functions
$x\mapsto f(y)-c(x,y)/\lambda$ have uniformly bounded first
derivatives and Hessians bounded below by $-2g/\lambda$, after
shrinking $r$ if necessary. Their supremum is locally Lipschitz
with a uniform constant and uniformly locally semiconvex. Integration
along geodesics makes the Lipschitz bound global within each
component.

For the maximum principle, observe that
\begin{equation*}
 \{S_\lambda f>t\}
 =\bigcup_{\{y:f(y)>t\}}
 B_g\!\left(y,\sqrt{2\lambda(f(y)-t)}\right).
\end{equation*}
Every point of this union is joined, inside the indicated ball, to
its center $y$. Its component in the union therefore contains $y$
and the entire component of $\{f>t\}$ containing $y$, since
$S_\lambda f\ge f$. The latter component is not relatively compact
by Lemma~\ref{lem1}. Consequently no component of the union is
relatively compact. Another application of Lemma~\ref{lem1}
finishes the proof.
\end{proof}

\begin{lemma}\label{lem5}
Suppose that $F:M\to\R$ is $\Lambda$-Lipschitz and uniformly locally
$C$-semiconvex. For all sufficiently small $\mu>0$, the following
properties hold.
\begin{enumerate}[label=\textup{(\roman*)}]
\item Each infimum $G(x)=I_\mu F(x)$ has a unique minimizer $P(x)$,
and $d_g(x,P(x))\le2\mu \Lambda$.
\item $G\in C^{1,1}_{\loc}(M)$ and
\begin{equation}\label{eq9}
 \nabla G(x)=-\mu^{-1}\exp_x^{-1}P(x).
\end{equation}
\item For every $\zeta\in\partial F(y)$, writing $z=\exp_y(\mu \zeta)$,
one has
\begin{equation}\label{eq10}
 P(z)=y,\qquad G(z)=F(y)+\tfrac\mu2|\zeta|^2.
\end{equation}
\item $0\le F-G\le\mu\Lambda^2/2$ on $M$.
\end{enumerate}
\end{lemma}

\begin{proof}
The assertion is immediate for $\Lambda=0$, component by component, so
assume $\Lambda>0$. Choose a uniform radius $r>0$ so that $F$ is
$C$-semiconvex on all relevant balls of radius $2r$, those balls
are strongly convex, and~\eqref{eq6} and the mixed-derivative
bounds hold there. Choose
\begin{equation*}
 2\mu \Lambda<r/2,\qquad \mu C<1/4.
\end{equation*}
The function
$H_x(y)=F(y)+c(x,y)/\mu$ is coercive, because $F$ is Lipschitz;
thus it attains its minimum. Comparison with $y=x$ gives
$d_g(x,P(x))\le2\mu \Lambda$ for any minimizer. On $B_g(x,2r)$,
$H_x$ is strongly geodesically convex, with convexity constant
\[
 a_\mu=\frac1{2\mu}-C\ge\frac1{4\mu}.
\]
All minimizers lie in this ball, which proves uniqueness.

The minimizer map $P$ is continuous. Indeed, if $x_k\to x$, all
$P(x_k)$ lie in one compact ball. Every convergent subsequence
minimizes $H_x$, and uniqueness identifies its limit as $P(x)$.
On a sufficiently small coordinate neighborhood, the defining infimum
may be restricted to one compact region where $c$ is smooth. The
functions $x\mapsto F(y)+c(x,y)/\mu$ have a common local upper
Hessian bound there, so their infimum $G$ is locally semiconcave.
More explicitly, comparison with $P(x)$ and $P(x+h)$ gives
\[
 \frac{c(x+h,P(x+h))-c(x,P(x+h))}{\mu}
 \le G(x+h)-G(x)
 \le\frac{c(x+h,P(x))-c(x,P(x))}{\mu}.
\]
The continuity of $P$ and a uniform first-order expansion of $c$
show that both bounds equal
$\mu^{-1}D_xc(x,P(x))h+o(|h|)$. Thus $G$ is differentiable,
its derivative is continuous, and~\eqref{eq9} follows.

Here is also the quantitative step giving $C^{1,1}$. Take nearby
$x,z$, and put $y=P(x)$ and $w=P(z)$. Strong convexity and the
minimizing property give
\[
 a_\mu d_g(y,w)^2
 \le\frac1\mu\{c(x,w)-c(x,y)+c(z,y)-c(z,w)\}.
\]
Integrating the mixed derivative of $c$ on the product of the two
short geodesic segments bounds the expression in braces in absolute
value by $C_0d_g(x,z)d_g(y,w)$. Consequently
\[
 d_g(P(x),P(z))\le4C_0d_g(x,z).
\]
The local derivative bounds for squared distance and~\eqref{eq9}
then make $\nabla G$ locally Lipschitz.

For (iii), let $\zeta\in\partial F(y)$ and $z=\exp_y(\mu \zeta)$.
Since $|\zeta|\le \Lambda$, the distance $d_g(z,y)$ is at most $\mu \Lambda$.
Moreover,
\[
 \nabla_y c(z,y)=-\exp_y^{-1}z=-\mu \zeta,
 \qquad 0\in\partial\big(F+c(z,\cdot)/\mu\big)(y).
\]
Strong convexity shows that $y$ minimizes this function on
$B_g(z,2r)$. Every global minimizer is in that ball by (i), so
$P(z)=y$. Substitution proves~\eqref{eq10}. Finally,
\[
 F(y)+\frac{d_g(x,y)^2}{2\mu}
 \ge F(x)-\Lambda d_g(x,y)+\frac{d_g(x,y)^2}{2\mu}
 \ge F(x)-\frac{\mu\Lambda^2}{2},
\]
while $G(x)\le F(x)$, proving (iv).
\end{proof}

\begin{lemma}\label{lem6}
Under the assumptions of Lemma~\ref{lem5}, if $F$ satisfies the
maximum principle, then $I_\mu F$ satisfies the maximum principle
for all sufficiently small $\mu>0$.
\end{lemma}

\begin{proof}
Fix $\mu$ for which Lemma~\ref{lem5} holds and write $G=I_\mu F$.
Suppose, to the contrary, that for some $V\Subset M$,
\[
 m:=\max_{\overline V}G>\max_{\partial V}G.
\]
The nonempty set $K=\{x\in\overline V:G(x)=m\}$ is compact and
contained in $V$. At every $x\in K$, $\nabla G(x)=0$, hence
\[
 P(x)=x,\qquad F(x)=G(x)=m,\qquad 0\in\partial F(x).
\]
If $\zeta\ne0$ belonged to $\partial F(x)$, convexity of the
subdifferential would imply $s \zeta\in\partial F(x)$ for
$0<s<1$. Formula~\eqref{eq10} would give
\[
 G(\exp_x(\mu s \zeta))=m+\tfrac\mu2s^2|\zeta|^2>m.
\]
For small $s$ this point lies in $V$, a contradiction. Therefore
\begin{equation}\label{eq11}
 \partial F(x)=\{0\}\qquad(x\in K).
\end{equation}

Upper semicontinuity of the subdifferential, compactness of $K$,
and~\eqref{eq11} give an open neighborhood $U$ of $K$, with
$\overline U\subset V$, small enough that
\[
 \exp_y(\mu \zeta)\in V
 \quad\text{whenever }y\in U,\ \zeta\in\partial F(y).
\]
Using~\eqref{eq10} once more yields
\begin{equation}\label{eq12}
 F(y)+\tfrac\mu2|\zeta|^2
 =G(\exp_y(\mu \zeta))\le m
 \qquad(y\in U,\ \zeta\in\partial F(y)).
\end{equation}
Thus $F\le m$ on $U$. If $F(y)=m$ there, any subgradient in
\eqref{eq12} must be zero. Formula~\eqref{eq10} then gives
$G(y)=m$, so $y\in K$. We have proved
\[
 \{y\in U:F(y)=m\}=K.
\]
Choose an open $W$ with $K\subset W\Subset U$. On the compact
boundary of $W$, $F<m$, whereas $F=m$ on $K$. This contradicts
the maximum principle for $F$.
\end{proof}

\begin{proposition}\label{prop2}
Let $\Omega\subset\R^n$ be bounded and open. If $f\in C(\Omega)$
satisfies the maximum principle, then for every $\eta>0$ there is
$a\in C^{1,1}_{\loc}(\Omega)$ satisfying the maximum principle and
\[
 \sup_\Omega|a-f|<\eta.
\]
The analogous assertion holds with the minimum principle.
\end{proposition}

\begin{proof}
Choose the metric from Lemma~\ref{lem3}. Lemma~\ref{lem4} gives
$F=S_\lambda f$ satisfying the maximum principle, with
$\|F-f\|_\infty<\eta/2$, for a sufficiently small fixed $\lambda$.
This $F$ is Lipschitz and uniformly locally semiconvex. Choose $\mu$
as in Lemmas~\ref{lem5} and~\ref{lem6}, and small enough that
$\mu\Lambda^2/2<\eta/2$. Then $a=I_\mu F$ has all the required
properties. Local $C^{1,1}$ regularity is unchanged by returning
to Euclidean coordinates. Applying the result to $-f$ and then
changing sign proves the minimum-principle assertion.
\end{proof}

\begin{remark}\label{rem3}
We do not assume boundedness of $f$. Instead, Lemma~\ref{lem4}
uses uniform continuity and the linear growth bound~\eqref{eq8}
to prove finiteness, uniform localization, and uniform approximation
directly. Lemma~\ref{lem5} treats the subsequent envelope.
Thus the boundedness hypothesis in the global theorem of Azagra
and Ferrera~\cite{AF2015} is not needed for these separate arguments.
\end{remark}

\section{Monotone double-obstacle minimization}\label{sec5}

We record the precise classical regularity input. For the
$p$-Dirichlet integral, $1<p<\infty$, a locally bounded solution
of a double-obstacle variational inequality with
$C^{1,1}_{\loc}$ obstacles belongs to
$C^{1,\alpha}_{\loc}$ for some $\alpha=\alpha(n,p)>0$.
The local estimates depend on the obstacle norms; the exponent may
be chosen only in terms of $n,p$. This is the $p$-Laplacian case
of Lieberman's interior regularity theorem~\cite{L1991}.
Using $C^{1,1}$ obstacles avoids any issue of choosing different
H\"older exponents at different points. Only interior regularity
is used.

\begin{proposition}\label{prop3}
Let $\Omega\subset\R^n$ be bounded and open, $1<p<\infty$, and
$u\in W^{1,p}(\Omega)$. Suppose that
$L,R\in C^{1,1}_{\loc}(\Omega)$ satisfy
\[
 L\le u\le R\quad\text{a.e.},\qquad L<R\quad\text{everywhere}.
\]
Assume that $L$ satisfies the maximum principle and $R$ the minimum
principle. Then the unique minimizer $v$ of $\Ep$ in
\begin{equation}\label{eq13}
 \mathcal K(L,R;u)=
 \{w\in u+W^{1,p}_0(\Omega):L\le w\le R\ \text{a.e.}\}
\end{equation}
has a monotone representative in $C^{1,\alpha}_{\loc}(\Omega)$,
where $\alpha=\alpha(n,p)>0$. In particular,
\[
 v-u\in W^{1,p}_0(\Omega),\qquad \Ep(v)\le\Ep(u).
\]
\end{proposition}

\begin{proof}
The class in~\eqref{eq13} is nonempty, since it contains $u$.
It is convex and norm closed, hence weakly closed. Poincar\'e's
inequality on $W^{1,p}_0(\Omega)$ makes an energy-bounded sequence
in $u+W^{1,p}_0(\Omega)$ bounded in $W^{1,p}(\Omega)$.
Reflexivity and lower semicontinuity give a minimizer.
Strict convexity of $|\xi|^p$ implies equality of the gradients of
any two minimizers, and Poincar\'e's inequality then gives equality
of the functions.

Local variations show that $v$ solves the double-obstacle
variational inequality on every relatively compact ball. The
obstacles are bounded and $C^{1,1}$ on such balls, and $v$ is
bounded between them. Lieberman's theorem~\cite{L1991} therefore
gives a representative in $C^{1,\alpha}_{\loc}(\Omega)$.
The a.e. inequalities $L\le v\le R$ hold everywhere for this
representative. No global Sobolev norm of the obstacles is needed
for this local application.

It remains to prove monotonicity. Fix $V\Subset\Omega$ and put
$m=\max_{\partial V}v$. Since $L\le v\le m$ on $\partial V$,
the maximum principle for $L$ implies $L\le m$ on $V$.
For $\delta>0$, define
\[
 w(x)=
 \begin{cases}
  \min\{v(x),m+\delta\},&x\in V,\\
  v(x),&x\notin V.
 \end{cases}
\]
Continuity of $v$ implies that
$(v-m-\delta)_+|_V$, extended by zero, has compact support in $V$
and belongs to $W^{1,p}_0(\Omega)$. Thus $w$ has the prescribed
Sobolev boundary values. Moreover $L\le w\le R$: the only new
value introduced is $m+\delta\ge L$, and the truncation only
lowers $v$. Finally $\Ep(w)\le\Ep(v)$. Uniqueness forces $w=v$,
and letting $\delta\downarrow0$ proves the maximum principle
for $v$. The same argument truncating from below, using the
minimum principle for $R$, proves the minimum principle.
\end{proof}

\begin{proof}[Proof of Theorem~\ref{thm1}]
Choose $\varepsilon_j\downarrow0$. Apply Proposition~\ref{prop2}
to $u$ with its maximum principle, and separately with its minimum
principle, obtaining $a_j,b_j\in C^{1,1}_{\loc}(\Omega)$ such that
\[
 \|a_j-u\|_\infty<\varepsilon_j/4,\qquad
 \|b_j-u\|_\infty<\varepsilon_j/4.
\]
Here $a_j$ satisfies the maximum principle and $b_j$ the minimum
principle. Put
\begin{equation*}
 L_j=a_j-\varepsilon_j/2,\qquad R_j=b_j+\varepsilon_j/2.
\end{equation*}
Addition of constants preserves the respective principles, and
\begin{equation}\label{eq14}
 u-\tfrac34\varepsilon_j<L_j<u-\tfrac14\varepsilon_j
 <u+\tfrac14\varepsilon_j<R_j<u+\tfrac34\varepsilon_j.
\end{equation}
Let $u_j$ be the minimizer from Proposition~\ref{prop3} for these
obstacles and boundary data $u$. That proposition proves
monotonicity and (A), (D), (E). The bounds~\eqref{eq14} give (B)
and, because $|\Omega|<\infty$, convergence in $L^p(\Omega)$.

The gradients are bounded in $L^p$ by (D). Every weakly convergent
subsequence has limit $\nabla u$, as follows from distributional
integration by parts and $u_j\to u$ in $L^p$. Hence the entire
gradient sequence converges weakly to $\nabla u$. Lower
semicontinuity and (D) now yield
\[
 \|\nabla u\|_{L^p}
 \le\liminf_j\|\nabla u_j\|_{L^p}
 \le\limsup_j\|\nabla u_j\|_{L^p}
 \le\|\nabla u\|_{L^p}.
\]
Uniform convexity of $L^p$, $1<p<\infty$, implies strong
convergence of the gradients. Together with convergence of the
functions this proves (C).
\end{proof}

\subsection{Strict decrease away from energy minimizers}

The strict separation in \eqref{eq14} gives a useful additional
conclusion. It will connect the arbitrary-dimensional construction
to the planar smooth approximation theorem.

\begin{corollary}\label{cor1}
For the sequence constructed in the proof of Theorem~\ref{thm1},
if $u$ is not $p$-harmonic in $\Omega$, then
\[
 \Ep(u_j;\Omega)<\Ep(u;\Omega)\qquad\text{for every }j.
\]
If $u$ is $p$-harmonic, every approximant with the boundary
and energy properties of that theorem equals $u$.
\end{corollary}

\begin{proof}
Suppose first that $\Ep(u_j;\Omega)=\Ep(u;\Omega)$ for one $j$.
Both $u$ and $u_j$ then minimize the energy over
$\mathcal K(L_j,R_j;u)$. Uniqueness in
Proposition~\ref{prop3} gives $u_j=u$ almost everywhere.
For any $\varphi\in C_c^\infty(\Omega)$, the uniform positive
gap between $u$ and the obstacles in \eqref{eq14} makes
$u+t\varphi$ admissible for all sufficiently small positive
and negative $t$. Its first variation at zero must therefore
vanish:
\[
 \int_\Omega |\nabla u|^{p-2}\nabla u\cdot\nabla\varphi\dd x=0.
\]
Since $\varphi$ was arbitrary, $u$ is $p$-harmonic.
This proves the strict assertion by contraposition. The last
assertion is Proposition~\ref{prop1}.
\end{proof}

\subsection{Why the obstacles must first be regularized}\label{subsec1}

It is tempting to omit the regularization in Section~\ref{sec4} and simply minimize between
$u-\varepsilon$ and $u+\varepsilon$. This construction does
preserve monotonicity and the energy bound, but it need not produce
even $C^1$ functions. The following two propositions make the
distinction explicit.

\begin{proposition}\label{prop4}
Let $u$ satisfy the assumptions of Theorem~\ref{thm1}. The minimizer
$T_\varepsilon u$ of $\Ep$ over
\[
 \{v\in u+W^{1,p}_0(\Omega):u-\varepsilon\le v\le u+\varepsilon\}
\]
has a continuous monotone representative. As $\varepsilon\downarrow0$
these minimizers converge to $u$ uniformly and strongly in $W^{1,p}$,
with unchanged Sobolev boundary values and no increase of energy.
\end{proposition}

\begin{proof}
Existence and uniqueness follow as in Proposition~\ref{prop3}.
Farnana~\cite[Theorem~3.9]{F2009}, applied in Euclidean space,
gives a continuous representative when both real-valued obstacles
are continuous. Its hypotheses hold for $u-\varepsilon$ and
$u+\varepsilon$. Both obstacles inherit the needed one-sided
principles from $u$. The truncation argument in
Proposition~\ref{prop3} applies without differentiability of the
obstacles. The error is bounded by $\varepsilon$, and the strong
convergence follows by exactly the uniform convexity argument in
the proof of Theorem~\ref{thm1}.
\end{proof}

\begin{proposition}\label{prop5}
Let $n\ge2$, $\Omega=B(0,2)$, $p=2$, and
\[
 u(x)=x_1+|x_n|.
\]
Put
\[
 c_n=\frac1{|\mathbb S^{n-1}|}
           \int_{\mathbb S^{n-1}}|\omega_n|\dd\mathcal H^{n-1}(\omega)>0.
\]
For $0<\varepsilon<c_n/2$, the function $T_\varepsilon u$ does
not belong to $C^1(B(0,1))$.
\end{proposition}

\begin{proof}
The function $u$ is Lipschitz and monotone, since on each line
parallel to $e_1$ it is strictly increasing. More explicitly, the
two endpoints of the component of such a line in a relatively
compact open set give values below and above the value at an
interior point.

Write $v=T_\varepsilon u$. In either half-ball $\{x_n>0\}$ or
$\{x_n<0\}$, both obstacles are affine and harmonic. In a ball
compactly contained in one half, the harmonic replacement of $v$
lies between those obstacles by the maximum principle. It is
therefore an admissible energy-decreasing replacement. Uniqueness
implies that $v$ is harmonic in each half-ball.

If $v\in C^1(B(0,1))$, integration by parts on the two sides of
$\{x_n=0\}$ shows that the normal fluxes cancel. Thus $v$ is
distributionally harmonic throughout $B(0,1)$, and hence harmonic.
Write $\avg_E$ for integration over a sphere $E$ divided by its
surface measure. For every $0<r<1$, the mean-value property and
$\|v-u\|_\infty\le\varepsilon$ then imply
\[
 r c_n
 =\avg_{\partial B(0,r)}u-u(0)
 \le\left|\avg_{\partial B(0,r)}(u-v)\right|
       +|v(0)-u(0)|
 \le2\varepsilon.
\]
Letting $r\uparrow1$ contradicts $2\varepsilon<c_n$.
\end{proof}

\begin{remark}\label{rem4}
Proposition~\ref{prop5} is an obstruction to a particular
minimization procedure, not an obstruction to approximation.
For the same $u$ an explicit smooth sequence is available.
Let $\rho(x)=4-|x|^2$, $0<\delta<1/4$, and set
\[
 \sigma_\delta(x)=\sqrt{x_n^2+\delta^2\rho(x)^2},\qquad
 g_\delta(x)=x_1+\sigma_\delta(x).
\]
Then $g_\delta\in C^\infty(\Omega)$,
$\|g_\delta-u\|_\infty\le4\delta$, and
$\partial_1g_\delta\ge1-4\delta>0$, so $g_\delta$ is monotone.
It extends continuously with the same boundary values as $u$.
A direct computation gives
\[
 |\nabla \sigma_\delta|^2-1
 =\frac{-4\delta^2\rho x_n^2+
       \delta^2\rho^2(4\delta^2|x|^2-1)}
       {x_n^2+\delta^2\rho^2}\le0.
\]
Since $\partial_1\sigma_\delta$ is odd in $x_1$,
\[
 E_2(g_\delta)=|\Omega|+
       \int_\Omega|\nabla \sigma_\delta|^2\dd x
 \le2|\Omega|=E_2(u).
\]
The uniform gradient bound and a.e. gradient convergence show
$g_\delta\to u$ strongly in $W^{1,r}(\Omega)$ for every
$1\le r<\infty$. The zero boundary difference belongs to
$W^{1,r}_0(\Omega)$.
\end{remark}

\section{Planar gradients and regularized minimizers}\label{sec6}

From this section through Section~\ref{sec8}, the ambient
dimension is two. The extra structure is the quasiregularity of
the complex gradient, which provides local uniform convergence
of gradients and controls their topological degrees.

\subsection{Quasiregular gradients and the critical-point index}

\begin{definition}\label{def3}
A continuous mapping $F:V\to\R^2$, with
$F\in W^{1,2}_{\mathrm{loc}}(V,\R^2)$, is
\emph{$K$-quasiregular} if
\begin{equation*}
 \|DF(x)\|_{\mathrm{op}}^2\le K\det DF(x)
 \quad\text{for almost every }x\in V.
\end{equation*}
We allow constant mappings when using compactness statements.
\end{definition}

This is the planar analytic convention used in \cite{AIM2009}.
Nonconstant quasiregular maps are open and discrete and have positive
local degree.  One way to see these properties is through the
Sto\"ilow factorization into a holomorphic map after a
quasiconformal change of variable \cite{AIM2009}.

We shall use two established facts, with their sources made explicit.
First, by \cite{BI1987,M1988}, if $u$ is planar $p$-harmonic then
\begin{equation*}
 F_u=u_x-iu_y
\end{equation*}
is quasiregular or constant.  In particular a nonconstant solution has
a closed discrete critical set.  Second, \cite[Theorem~1]{HM2020}
says that a family of planar $K$-quasiregular maps with a uniform
$L^s(V)$ bound, for any $s>0$, is relatively compact for local uniform
convergence, and its limits are quasiregular maps or constants with
values in $\C$.  The latter assertion is stronger than compactness merely
in the spherical metric and is the version needed here.

\subsubsection*{Index and Morse critical points}

For a continuous vector field $X$ with an isolated zero $x_0$, its
index is the degree of $X/|X|$ on a small positively oriented circle
about $x_0$.  We use the standard degree and homotopy conventions of
\cite{M1965}; for applications to planar elliptic equations see
\cite{AM1992,AM1994}.  Since complex conjugation reverses orientation,
\begin{equation*}
 \ind_{x_0}(\nabla u)=-\deg(F_u,B_r(x_0),0).
\end{equation*}
At a critical point of a nonconstant planar $p$-harmonic function the
degree on the right is a positive integer.  This proves the sign
assertion in Theorem~\ref{thm3}.

A critical point of a smooth real function is \emph{nondegenerate}
if its Hessian is invertible.  A function all of whose critical
points are nondegenerate is a Morse function.  In two dimensions a
critical point with negative Hessian determinant is a saddle and has
gradient index $-1$.  By the Morse lemma it has local smooth coordinates
in which the function is its critical value plus $s^2-t^2$;
see \cite{M1963}.  Sard's theorem \cite{S1942} will be applied to
the smooth map $\nabla v:\R^2\supset B\to\R^2$ in order to make
its zeros nondegenerate after a small translation of the target.

\subsection{Regularized minimizers and \texorpdfstring{$C^1$}{C1} convergence}\label{subsec2}

Let $B=B_R(x_0)\Subset V$, and let $u$ be $p$-harmonic on a
neighborhood of $\overline B$.  For $0<\varepsilon\le1$ put
\begin{equation*}
 J_\varepsilon(v;B)
   =\int_B(\varepsilon^2+|\nabla v|^2)^{p/2}\dd x.
\end{equation*}
The admissible class will always be $u+W^{1,p}_0(B)$.

\begin{lemma}\label{lem7}
There is a unique minimizer $v_\varepsilon$ of
$J_\varepsilon(\cdot;B)$ in $u+W^{1,p}_0(B)$.  It belongs to
$C^\infty(B)$ and solves
\begin{equation}\label{eq15}
 \diver\bigl((\varepsilon^2+|\nabla v_\varepsilon|^2)^{(p-2)/2}
                    \nabla v_\varepsilon\bigr)=0.
\end{equation}
If $u$ is nonconstant in $B$, then $v_\varepsilon$ is nonconstant.
\end{lemma}

\begin{proof}
The integrand is strictly convex in its gradient variable.  Its Hessian,
apart from the positive scalar factor
$p(\varepsilon^2+|\xi|^2)^{(p-2)/2}$, is
\begin{equation*}
 I+(p-2)\frac{\xi\otimes\xi}{\varepsilon^2+|\xi|^2}.
\end{equation*}
Its eigenvalues lie between $\min\{1,p-1\}$ and
$\max\{1,p-1\}$.  Coercivity follows from
$J_\varepsilon(v;B)\ge\|\nabla v\|_{L^p(B)}^p$ and Poincar\'e's
inequality for $v-u$.  The direct method gives a minimizer, and strict
convexity together with the boundary condition gives uniqueness.
The first variation is \eqref{eq15}.

For fixed $\varepsilon>0$ the vector field in \eqref{eq15} has the
standard $p$-growth and ellipticity structure, with weight
$(\varepsilon^2+|\xi|^2)^{(p-2)/2}$.  The scalar interior regularity
theorem in \cite{T1984} gives
$v_\varepsilon\in C^{1,\alpha}_{\mathrm{loc}}(B)$.
Its gradient is therefore bounded on each compact subdisk.  There the
coefficient matrix is uniformly positive definite and is a smooth
function of the gradient.  The usual quasilinear interior Schauder
bootstrap gives $C^\infty$ regularity; see \cite{GT2001}.  Only
fixed parameter regularity is used in this step.

In nondivergence form the smooth solution satisfies
\begin{equation}\label{eq16}
 \tr(A_\varepsilon D^2v_\varepsilon)=0,
 \qquad
 A_\varepsilon
 =I+(p-2)\frac{\nabla v_\varepsilon\otimes\nabla v_\varepsilon}
                    {\varepsilon^2+|\nabla v_\varepsilon|^2}.
\end{equation}
The matrix $A_\varepsilon$ has the uniform ellipticity bounds just
given.  If $v_\varepsilon=c$, then $u-c\in W^{1,p}_0(B)$.
Testing the $p$-harmonic equation for $u$ with $u-c$ gives
$\int_B|\nabla u|^p=0$.  Thus $u$ is constant, contrary to the
hypothesis.
\end{proof}

\begin{lemma}\label{lem8}
As $\varepsilon\downarrow0$, the minimizers satisfy
\begin{equation}\label{eq17}
 v_\varepsilon\longrightarrow u
 \quad\text{strongly in }W^{1,p}(B).
\end{equation}
\end{lemma}

\begin{proof}
Minimality yields
\begin{equation}\label{eq18}
 \int_B|\nabla v_\varepsilon|^p
 \le J_\varepsilon(v_\varepsilon;B)
 \le J_\varepsilon(u;B)
 \longrightarrow\Ep(u;B).
\end{equation}
The last convergence is dominated convergence, using
$(\varepsilon^2+|\nabla u|^2)^{p/2}
\le C_p(1+|\nabla u|^p)$ for $0<\varepsilon\le1$.
Poincar\'e's inequality makes $v_\varepsilon$ bounded in $W^{1,p}(B)$.
Every weak cluster point $v$ has $v-u\in W^{1,p}_0(B)$ and satisfies
\[
 \Ep(v;B)\le\liminf_{\varepsilon\downarrow0}
             \Ep(v_\varepsilon;B)\le\Ep(u;B).
\]
The $p$-harmonic function $u$ is the unique minimizer of $\Ep$ in
this affine class, by strict convexity.  Hence every weak cluster
point is $u$.  Also $\Ep(v_\varepsilon;B)\ge\Ep(u;B)$, so
\eqref{eq18} gives convergence of the $L^p$ norms of the gradients.
Uniform convexity of $L^p$, $1<p<\infty$, now gives strong convergence
of the gradients; see \cite{B2011}.  Poincar\'e's inequality for
$v_\varepsilon-u$ proves \eqref{eq17}.
\end{proof}

The next elementary matrix calculation makes the parameter independence
of the quasiregular bound explicit.

\begin{lemma}\label{lem9}
Set
\begin{equation*}
 K_p=\max\{p-1,(p-1)^{-1}\}.
\end{equation*}
For every $\varepsilon>0$ the complex gradient
$F_\varepsilon=(v_\varepsilon)_x-i(v_\varepsilon)_y$ is
$K_p$-quasiregular or constant in $B$.
\end{lemma}

\begin{proof}
Write $Q=D^2v_\varepsilon$ at a fixed point.  If $Q\ne0$, the
identity $\tr(A_\varepsilon Q)=0$ and positivity of $A_\varepsilon$
force the two eigenvalues of $Q$ to have opposite signs.  Write them
as $h>0$ and $-k<0$.  In an orthonormal eigenbasis of $Q$ the
identity becomes $a_{11}h=a_{22}k$, where
\[
 \min\{1,p-1\}\le a_{11},a_{22}\le\max\{1,p-1\}.
\]
It follows that $K_p^{-1}\le h/k\le K_p$.  Identifying
$F_\varepsilon$ with the real map
$((v_\varepsilon)_x,-(v_\varepsilon)_y)$, we have
\[
 DF_\varepsilon=\begin{pmatrix}1&0\\0&-1\end{pmatrix}Q,
 \qquad J_{F_\varepsilon}=-\det Q=hk.
\]
Therefore
\begin{equation*}
 \|DF_\varepsilon\|_{\mathrm{op}}^2
   =\max\{h^2,k^2\}\le K_p hk=K_pJ_{F_\varepsilon}.
\end{equation*}
If $Q=0$, the same inequality is immediate.  Since
$F_\varepsilon$ is smooth, the Sobolev requirement in
Definition~\ref{def3} also holds.
\end{proof}

\begin{lemma}\label{lem10}
If $u$ is nonconstant in $B$, then $v_\varepsilon$ has no local
maximum or local minimum in $B$.
\end{lemma}

\begin{proof}
At a local maximum the strong maximum principle applied to
\eqref{eq16} on a sufficiently small disk makes $v_\varepsilon$
constant on that disk.  Its complex gradient then vanishes on a
nonempty open set.  By Lemma~\ref{lem9}, this gradient is
quasiregular or constant.  A nonconstant quasiregular map has discrete
point preimages, so in either case the gradient must vanish identically
in $B$.  This contradicts the nonconstancy assertion of
Lemma~\ref{lem7}.  The argument at a local minimum is identical.
\end{proof}

\begin{proposition}\label{prop6}
For every $0<s<R$,
\begin{equation}\label{eq19}
 \|v_\varepsilon-u\|_{C^1(\overline{B_s(x_0)})}
 \longrightarrow0
 \quad\text{as }\varepsilon\downarrow0.
\end{equation}
\end{proposition}

\begin{proof}
By \eqref{eq18}, $F_\varepsilon$ is uniformly bounded in $L^p(B)$.
Lemma~\ref{lem9} and \cite[Theorem~1]{HM2020} show that every sequence
$\varepsilon_j\downarrow0$ has a subsequence for which
$F_{\varepsilon_j}$ converges locally uniformly to a finite-valued
continuous map $F$.  Lemma~\ref{lem8} identifies this limit almost
everywhere with $u_x-iu_y$.  Both maps are continuous, so the equality
holds everywhere.  Every convergent subsequence has the same limit;
the compactness assertion therefore gives
\[
 \nabla v_\varepsilon\longrightarrow\nabla u
 \quad\text{locally uniformly in }B.
\]
For completeness, put $h_\varepsilon=v_\varepsilon-u$.  On a compact
subdisk $B_s$ the line segment inequality gives
\[
 \sup_{B_s}|h_\varepsilon|
 \le |B_s|^{-1/p}\|h_\varepsilon\|_{L^p(B_s)}
       +2s\sup_{B_s}|\nabla h_\varepsilon|.
\]
Both terms tend to zero, proving \eqref{eq19}.
\end{proof}

\begin{remark}\label{rem5}
The proof uses classical interior regularity to make each fixed
$v_\varepsilon$ smooth, and quasiregular compactness to control the
family as $\varepsilon$ tends to zero.  No estimate for the boundary
gradient of $v_\varepsilon$ is required.  All joining operations below
take place in a compact subdisk of $B$.
\end{remark}

\section{Local replacement and resolution of critical points}\label{sec7}

\subsection{Proof of the local smoothing theorem}

\begin{proof}[Proof of Theorem~\ref{thm2}]
Choose $r>0$ and $R>2r$ so that
$\overline{B_{2r}(x_0)}\subset O$ and
$\overline{B_R(x_0)}\subset V$. Since $\Crit(u)=\{x_0\}$,
$u$ is nonconstant on $B_R(x_0)$; otherwise its gradient would
vanish throughout that disk. Let
$\chi\in C_c^\infty(B_{2r}(x_0))$ satisfy
$0\le\chi\le1$ and $\chi=1$ on $B_r(x_0)$.  Construct
$v_\varepsilon$ in $B_R(x_0)$ as in Subsection~\ref{subsec2}, and set
\begin{equation}\label{eq20}
 w_\varepsilon=u+\chi(v_\varepsilon-u)
 \quad\text{in }B_R(x_0),
 \qquad w_\varepsilon=u\quad\text{outside }B_R(x_0).
\end{equation}
Near $x_0$ this is $v_\varepsilon$.  Away from $x_0$, the original
function is smooth.  Thus $w_\varepsilon\in C^\infty(V)$ and
$w_\varepsilon-u$ has compact support in $O$.

On the closed annulus
$A=\overline{B_{2r}(x_0)}\setminus B_r(x_0)$, the only-critical-point
assumption gives
\begin{equation*}
 c_A:=\min_A|\nabla u|>0.
\end{equation*}
The gradient of the difference is
\begin{equation}\label{eq21}
 \nabla(w_\varepsilon-u)
 =\chi(\nabla v_\varepsilon-\nabla u)
     +(v_\varepsilon-u)\nabla\chi.
\end{equation}
Proposition~\ref{prop6} implies that this converges uniformly to zero
on $\overline{B_{2r}(x_0)}$.  For small enough $\varepsilon$,
\begin{equation}\label{eq22}
 |\nabla w_\varepsilon|\ge c_A/2\quad\text{on }A.
\end{equation}
Inside $B_r(x_0)$, the function $w_\varepsilon=v_\varepsilon$ has no
local extrema by Lemma~\ref{lem10}.  On $A$ and outside
$B_{2r}(x_0)$ its gradient does not vanish.  Hence there is no local
extremum anywhere in $V$, including the two joining circles.

Equations \eqref{eq20} and \eqref{eq21} imply convergence to zero of
the $C^1(V)$ norm of the difference.  Since that difference is supported
in a fixed disk of finite area, its $W^{1,p}(V)$ norm also tends to
zero.  Choose $\varepsilon$ to obtain \eqref{eq1}.  A compactly
supported $W^{1,p}$ function in $V$ lies in $W^{1,p}_0(V)$, by
interior mollification.  This proves all the assertions.
\end{proof}

\begin{remark}\label{rem6}
The joining estimate uses both terms in \eqref{eq21}.  Uniform
convergence of the values and strong convergence of the gradients in
$L^p$ alone would not provide \eqref{eq22}.  The local uniform
convergence of the gradients is the additional information supplied by
Proposition~\ref{prop6}.
\end{remark}

\subsection{Proof of the saddle resolution theorem}

\begin{proof}[Proof of Theorem~\ref{thm3}]
Use the disks and cutoff from the preceding proof, decreasing $r$ if
necessary.  Proposition~\ref{prop6} implies, for sufficiently small
$\varepsilon$, that $\nabla v_\varepsilon$ is nonzero on
$\partial B_r(x_0)$ and that
\begin{equation}\label{eq23}
 \deg(\nabla v_\varepsilon,B_r(x_0),0)
   =\deg(\nabla u,B_r(x_0),0)=-m.
\end{equation}
Indeed the straight homotopy between the two gradients stays away
from zero on the boundary circle.

Choose a vector $a\in\R^2$, as small as needed, such that $-a$ is a
regular value of the smooth mapping $\nabla v_\varepsilon$ on
$B_r(x_0)$.  Such vectors exist arbitrarily close to zero by Sard's
theorem \cite{S1942}.  Also choose $|a|$ smaller than half the minimum
of $|\nabla v_\varepsilon|$ on the boundary circle.  Set
\[
 z_{\varepsilon,a}(x)=v_\varepsilon(x)+a\cdot(x-x_0).
\]
Every critical point of $z_{\varepsilon,a}$ in $B_r(x_0)$ is
nondegenerate.  The coefficient matrix $A_\varepsilon(x)$ in
\eqref{eq16}, now regarded as a given matrix field, satisfies
\begin{equation}\label{eq24}
 \tr(A_\varepsilon D^2z_{\varepsilon,a})=0,
\end{equation}
because the added function is affine.  A nonzero positive or negative
semidefinite Hessian cannot satisfy \eqref{eq24}.  At a nondegenerate
critical point the Hessian is consequently indefinite, and the point
is a saddle of index $-1$.

The zeros are finite in number: they are isolated by nondegeneracy,
and they lie in a compact subset of $B_r(x_0)$ since the gradient is
nonzero near its boundary.  Degree is unchanged by the small
translation $a$.  Additivity of degree and \eqref{eq23} therefore
give
\[
 -m=\sum_{\nabla z_{\varepsilon,a}(q)=0}
             \operatorname{sign}\det D^2z_{\varepsilon,a}(q)
       =-\#\Crit(z_{\varepsilon,a}|_{B_r(x_0)}).
\]
There are exactly $m$ such points.

Now define
\begin{equation}\label{eq25}
 w=u+\chi\bigl(v_\varepsilon-u+a\cdot(x-x_0)\bigr).
\end{equation}
First choose $\varepsilon$ small and then choose the regular-value
vector $a$ sufficiently small.  The estimate \eqref{eq22} still holds
on the joining annulus.  Thus the only critical points of $w$ are the
$m$ saddles inside $B_r(x_0)$.  The nondegenerate saddle form
excludes local extrema at these points, and the nonzero gradient
excludes local extrema everywhere else.
The resulting $w$ is strictly monotone.  Its support and norm
properties follow as in Theorem~\ref{thm2}.

Finally choose pairwise disjoint small disks about the $m$ saddles and
smooth cutoff functions $\theta_i$ supported in those disks and equal
to one near the corresponding saddle $q_i$.  The gradient of $w$ has
a positive lower bound on the compact sets where $\nabla\theta_i$
is supported.  For sufficiently small constants $b_i$, the function
$w+\sum_i b_i\theta_i$ has precisely the same critical points and
the same Hessians there.  The constants can be chosen to make
$w(q_i)+b_i$ pairwise distinct.  No new critical point is introduced,
and the additional perturbation can be arbitrarily small in $C^1$
and $W^{1,p}$.  Starting with a smaller error budget in
\eqref{eq25} proves the final assertion.
\end{proof}

\begin{remark}\label{rem7}
The number $m$ is forced once all replacement critical points are
nondegenerate saddles and the gradient on the surrounding circle is
fixed.  The theorem does not prescribe their positions or their
individual critical values, nor does it preserve every level set up
to homeomorphism.  The total gradient index is the preserved invariant.
\end{remark}

\section{Global smooth density and the energy condition}\label{sec8}

\subsection{Locally finite replacement}

The following proposition describes exactly the passage from the local
theorem to the global approximants of \cite{N2020a}.

\begin{proposition}\label{prop7}
Let $\Omega\subset\R^2$ be open and let
$v\in C^1(\Omega)\cap W^{1,p}(\Omega)$ be monotone.  Suppose that
$S\subset\Omega$ is closed and discrete, $v$ is smooth in
$\Omega\setminus S$, and every $s\in S$ has a neighborhood in which
$v$ is nonconstant and $p$-harmonic.  For every $\eta,\delta>0$
there is a smooth monotone $w\in W^{1,p}(\Omega)$ with
\begin{equation}\label{eq26}
 \|w-v\|_{C^1(\Omega)}+\|w-v\|_{W^{1,p}(\Omega)}<\eta,
 \qquad w-v\in W^{1,p}_0(\Omega),
\end{equation}
such that $\{w\ne v\}$ is contained in an open set of area less
than $\delta$.  All new critical points introduced near $S$ may be
chosen nondegenerate saddles.
\end{proposition}

\begin{proof}
If $S$ is empty, take $w=v$. Otherwise enumerate
$S=\{s_1,s_2,\ldots\}$, with the evident change for a
finite set.  Any point of $S$ at which $\nabla v\ne0$ is already
smooth by elliptic regularity and can be omitted.  By isolation of
planar $p$-harmonic critical points, choose disks $D_i\Subset\Omega$
centered at $s_i$ such that $v$ is nonconstant and $p$-harmonic
in $D_i$ and has no other critical point there. Require explicitly
that the closures $\overline{D_i}$ are pairwise disjoint and form
a locally finite family in $\Omega$. To achieve this, take the
radius of $D_i$ smaller than one third of the distance from $s_i$
to $S\setminus\{s_i\}$ and also smaller than $1/i$, as well as
small enough for the preceding analytic requirements. The distance
condition is vacuous if $S$ is a singleton. If infinitely many such
disks met a fixed compact subset of $\Omega$, the radius bound
would force their centers to accumulate in that compact subset,
contrary to closed discreteness of $S$.
Choose smaller concentric disks $B_i\Subset D_i$ with
$\sum_i|B_i|<\delta$.

Apply Theorem~\ref{thm3} in each $D_i$, with the change supported in
$B_i$ and with error smaller than $\eta2^{-i-1}$. Denote the
compactly supported difference by $h_i$, extended by zero to
$\Omega$. For a finite replacement $w_N=v+\sum_{i=1}^Nh_i$, put
\[
 K_N=\bigcup_{i=1}^N\supp h_i,
 \qquad A_N=\bigcup_{i=1}^ND_i.
\]
Then $K_N\Subset A_N$ and $w_N=v$ outside $K_N$. On each $D_i$
with $i\le N$, disjointness gives $w_N=v+h_i$, exactly the output
of Theorem~\ref{thm3}, so $w_N$ has no local extrema in $A_N$.
Lemma~\ref{lem2} therefore proves that every $w_N$ is monotone.
The series converges uniformly and
in $W^{1,p}$, and its $C^1$ difference from $v$ has the bound in
\eqref{eq26}.  Local finiteness implies that
$w=v+\sum_i h_i$ is smooth at every point of $\Omega$.
The uniform limit assertion in Lemma~\ref{lem2} proves monotonicity.
Each $h_i$ belongs to $W^{1,p}_0(\Omega)$, and this space is closed
in $W^{1,p}$, proving the boundary assertion.  The change is supported
in $\bigcup_i B_i$, whose area is less than $\delta$.
\end{proof}

\subsection{Completion of the planar approximation theorem}

We specify the construction underlying~\cite[Theorem~1.6]{N2020a}.
In its Subsection~4.2.2(a),(c), possible nonsmooth points are critical
points of nonconstant $p$-harmonic pieces; they have no accumulation
in $\Omega$. Subsection~4.3 smooths the regular interfaces in
arbitrarily small, pairwise disjoint, locally finite neighborhoods.
Those neighborhoods can be chosen with locally finite closures
avoiding the closed discrete critical set. Hence the resulting
approximant remains nonconstant and $p$-harmonic near each possible
nonsmooth point. Removing these closed neighborhoods from the open
union of the $p$-harmonic pieces leaves an open set on which the
approximant is $p$-harmonic. The interfaces have area zero, so the
neighborhoods can be chosen to have arbitrarily small total area.

Thus the construction supplies $C^{1,\alpha}_{\loc}$ approximants
$v_j$, smooth off closed discrete sets $S_j$, and open sets $G_j^0$
with $|\Omega\setminus G_j^0|$ arbitrarily small, with precisely
the local $p$-harmonic properties needed in Proposition~\ref{prop7}.
Properties (B)--(E) of the cited theorem give uniform and strong
$W^{1,p}$ convergence, the fixed Sobolev boundary class, and
$\Ep(v_j;\Omega)\le\Ep(u;\Omega)$.

The first part of the proof below uses only the nonincreasing
energy conclusion of the quoted theorem. The strict inequality
for a non-$p$-harmonic input will then follow from
Corollary~\ref{cor1}, rather than from a separate choice of
level strips in the planar construction.

\begin{proof}[Proof of Theorem~\ref{thm4}]
Choose an approximant $v=v_j$ as above so that the sum of the uniform
and Sobolev errors from $u$ is less than $\eta/2$, and choose its
$p$-harmonic open set $G^0$ with
$|\Omega\setminus G^0|<\delta/2$.
Apply Proposition~\ref{prop7} to $v$, taking its error smaller than
$\eta/2$ and its replacement disks to have total area less than
$\delta/2$.  The resulting function $w$ is smooth and monotone and
satisfies \eqref{eq2}.  Its boundary difference is the sum of two
elements of $W^{1,p}_0(\Omega)$.

The map $z\mapsto\Ep(z;\Omega)$ is continuous in the strong
$W^{1,p}$ topology.  Thus the error in Proposition~\ref{prop7} may
also be chosen to ensure
$\Ep(w;\Omega)<\Ep(v;\Omega)+\eta$.
The union of the closed replacement disks is relatively closed in
$\Omega$, by local finiteness.  Remove that union from $G^0$.
The remaining set $G$ is open, $w=v$ there, and
$|\Omega\setminus G|<\delta$. This proves all assertions except
the strict energy improvement.

For that improvement, assume that $u$ is not $p$-harmonic.
Theorem~\ref{thm1} and Corollary~\ref{cor1} give a monotone
$a$ in the same boundary class such that
\[
 \|a-u\|_{L^\infty}+\|a-u\|_{W^{1,p}}<\eta/2,
 \qquad \gamma:=\Ep(u;\Omega)-\Ep(a;\Omega)>0.
\]
Apply the part of the present theorem already proved to $a$,
with the same $\delta$ and an error tolerance
$0<\tau<\min\{\eta/2,\gamma\}$.
The resulting smooth monotone $w$ has total approximation error
less than $\eta$, retains the boundary class and the large
$p$-harmonic open set, and satisfies
\[
 \Ep(w;\Omega)\le\Ep(a;\Omega)+\tau<\Ep(u;\Omega).
\]
This proves the strict assertion without circular use of it.
\end{proof}

\begin{corollary}\label{cor2}
Under the hypotheses of Theorem~\ref{thm4}, smooth monotone functions
with the same Sobolev boundary values approximate $u$ uniformly and
strongly in $W^{1,p}$.  Their $p$-energies converge to $\Ep(u;\Omega)$.
If $u$ itself is nonconstant and $p$-harmonic on a connected $\Omega$,
the approximants can also be Morse functions, with
$\|w_j-u\|_{C^1(\Omega)}\to0$, and their changes can be confined to
open sets whose areas tend to zero.
\end{corollary}

\begin{proof}
The first assertion follows by applying Theorem~\ref{thm4} with
$\eta\downarrow0$.  For the second, use Proposition~\ref{prop7}
directly on $u$, with $S=\Crit(u)$.  This set is closed and discrete,
and $u$ is smooth elsewhere.  Each replacement resolves the sole
critical point of its disk into nondegenerate saddles, and there are
no critical points outside these disks.  If $S$ is empty, take
$w_j=u$.
\end{proof}

\section{Nonmanifold levels in every higher dimension}\label{sec9}

We use the following elementary topological criterion. A loop is null-homotopic in a space
if it can be continuously deformed there to a constant loop.
The degree of a map from $\mathbb S^1$ to itself is its winding
number. In particular, a degree-one loop is not null-homotopic;
see Hatcher~\cite[Section~1.1]{H2002}.

\begin{lemma}\label{lem11}
Let $X$ be a subspace of a Euclidean space and let $P\in X$.
Suppose that the loops $\gamma_k:\Sph^1\to X$ converge uniformly
to the constant loop at $P$, and there are continuous maps $\rho_k:X\to\Sph^1$ such that
$\rho_k\circ\gamma_k$ has degree one. Then $P$ has no neighborhood
in $X$ homeomorphic to an open Euclidean ball or to a relatively
open ball in a Euclidean half-space.
\end{lemma}

\begin{proof}
If $\gamma_k$ were null-homotopic in $X$, composition with $\rho_k$
would contract a degree-one loop in $\Sph^1$, contradicting
$\pi_1(\Sph^1)\cong\Z$. Now suppose a manifold chart existed at
$P$. After restriction, its domain would be a relatively open,
contractible neighborhood $W$ of $P$ in $X$. Uniform convergence
to the constant loop at $P$ places the entire loop $\gamma_k$ in $W$ for
all sufficiently large $k$. Contracting it in $W$ gives the same
contradiction.
\end{proof}

\subsection{Separated radial perturbations}

Write points of $\R^n$ as $(x,y,z,\xi)$, where
$\xi\in\R^{n-3}$ is absent when $n=3$, and set $q=(y,z)$.
We first construct $w$ on the whole $(y,z)$-plane.

Choose $\psi\in C^\infty([0,\infty))$ with
\begin{equation}\label{eq27}
 \psi(s)=1-s^2\quad(0\le s\le1/2),\qquad
 \psi'(s)<0\quad(0<s<1),\qquad
 \psi(s)=0\quad(s\ge1).
\end{equation}
Such a choice is explicit. Put $\chi_0(t)=e^{-1/t}$ for $t>0$ and
$\chi_0(t)=0$ for $t\le0$, and define
\[
 \vartheta(s)=
 \frac{\chi_0(1-s)}{\chi_0(1-s)+\chi_0(s-1/2)},
 \qquad \psi(s)=(1-s^2)\vartheta(s)\quad(s\ge0).
\]
The denominator never vanishes. On $(1/2,1)$ the function
$\vartheta$ is positive and strictly decreasing. This verifies
\eqref{eq27}; the matching is smooth at both endpoints.
The radial function $q\mapsto\psi(|q|)$ is smooth at the origin
because it equals $1-|q|^2$ near that point. In particular,
$0\le\psi\le1$.

For $k\ge1$ and $j\in\Z$ define
\begin{equation*}
 r_k=2^{-k},\qquad a_{k,j}=8j r_k,\qquad
 c_{k,j}=(a_{k,j},a_{k,j}-56r_k),
\end{equation*}
and set
\begin{equation}\label{eq28}
 b_{k,j}(q)=64r_k\psi\!\left(\frac{|q-c_{k,j}|}{r_k}\right),
 \qquad
 b(q)=\sum_{k\ge1}\sum_{j\in\Z}b_{k,j}(q),
 \qquad w(y,z)=z+b(y,z).
\end{equation}
Finally, put
\begin{equation}\label{eq29}
 U(x,y,z,\xi)=x^2+w(y,z).
\end{equation}

\begin{proposition}\label{prop8}
The closed disks $\overline B(c_{k,j},r_k)$ are pairwise disjoint.
They form a locally finite family in the complement of
$\mathcal D=\{(s,s):s\in\R\}$. The function $b$ is globally Lipschitz,
vanishes on $\mathcal D$, and is smooth on $\R^2\setminus \mathcal D$.
\end{proposition}

\begin{proof}
For a fixed $k$, adjacent centers are separated by
$8\sqrt2\,r_k>2r_k$. The linear coordinate $y-z$ on the disk at
scale $k$ lies in
\[
 [(56-\sqrt2)r_k,(56+\sqrt2)r_k].
\]
If $\ell>k$, then $r_\ell\le r_k/2$, and
\[
 (56+\sqrt2)r_\ell
 \le(28+\sqrt2/2)r_k<(56-\sqrt2)r_k.
\]
This separates disks at different scales and also shows that none
meets $\mathcal D$. On any compact set disjoint from $\mathcal D$ only finitely many
scales occur, and at any one scale only finitely many centers occur.
This proves local finiteness and smoothness away from $\mathcal D$.

Each $b_{k,j}$, extended by zero as in \eqref{eq28}, has Lipschitz
constant at most
\[
 C_b=64\|\psi'\|_\infty.
\]
The supports are disjoint and all summands are nonnegative. Hence
\[
 b(q)=\sup_{k,j}b_{k,j}(q),\qquad 0\le b(q)\le32.
\]
A finite-valued supremum of functions with the same Lipschitz
constant has that Lipschitz constant. Thus
$|b(q)-b(q')|\le C_b|q-q'|$ globally, including at $\mathcal D$.
\end{proof}

\begin{proposition}\label{prop9}
The function $w$ has no local minimum. The function $U$ has no
local maximum and no local minimum. Its restriction to $(-2,2)^n$
is Lipschitz and monotone.
\end{proposition}

\begin{proof}
Suppose first that $q=(y,z)\in B(c_{k,j},r_k)$. Keep $z$ fixed
and change $y$ by an arbitrarily small amount in a direction that
increases $|y-a_{k,j}|$. If $y=a_{k,j}$, either direction works.
The distance to $c_{k,j}$ strictly increases while the point stays
inside the same disk. By \eqref{eq27}, the bump value strictly
decreases. No other bump contributes, so $w$ strictly decreases.

At a point outside all the open disks and outside $\mathcal D$, local
finiteness and the flatness of $\psi$ at $1$ give
$\nabla w=(0,1)$, including at the boundary of a disk. Such a
point is not a local minimum. At a point $(s,s)\in \mathcal D$ one has
$w(s,s)=s$ and $w(s-h,s-h)=s-h$ for every $h>0$. This excludes
local minima on $\mathcal D$ as well.

At every point, changing $x$ so as to increase $|x|$ increases
$U$. Thus $U$ has no local maximum. If $x\ne0$, moving $x$
towards zero decreases $U$. If $x=0$, a change in $q$ that
decreases $w$ decreases $U$. Thus $U$ also has no local minimum.
The strictness criterion following Definition~\ref{def1} proves monotonicity. Proposition~\ref{prop8}
and boundedness of $x$ prove the Lipschitz assertion on the cube.
\end{proof}

\subsection{A noncontractible circle at every level}

Fix $t\in\R$. At scale $k$ choose $j=j(t,k)\in\Z$ so that
\begin{equation}\label{eq30}
 |a_{k,j}-t|\le4r_k.
\end{equation}
The lattice spacing $8r_k$ makes this possible. In the next
calculation abbreviate $r=r_k$, $a=a_{k,j}$, and $c=c_{k,j}$.
On $\overline B(c,r/2)$, equations \eqref{eq27}--\eqref{eq28}
give the exact formula
\begin{align}
 w(q)
 &=a+8r+(z-c_z)-\frac{64}{r}|q-c|^2 \notag\\
 &=m-\frac{64}{r}|q-d|^2,
 \qquad
 d=c+(0,r/128),\qquad
 m=a+(8+1/256)r. \label{eq31}
\end{align}
The shift from $c$ to $d$ is necessary: adding the background
coordinate $z$ moves the maximum of the quadratic bump.

\begin{proposition}\label{prop10}
For each $t\in\R$ and each $k\ge1$ there is a disk
$D_{t,k}=B(d_{t,k},R_{t,k})$ such that
\begin{equation}\label{eq32}
 w=t\ \hbox{on }\partial D_{t,k},\qquad
 w>t\ \hbox{in }D_{t,k},\qquad
 \sup_{q\in\overline D_{t,k}}|q-(t,t)|<64r_k.
\end{equation}
In particular, $w(d_{t,k})>t$.
\end{proposition}

\begin{proof}
Use $d$ and $m$ from \eqref{eq31} and define
\[
 R=\left(\frac r{64}(m-t)\right)^{1/2}.
\]
By \eqref{eq30},
\[
 (4+1/256)r\le m-t\le(12+1/256)r,
\]
so $R>0$. Moreover,
\[
 |d-c|+R\le\frac{1+\sqrt{3073}}{128}\,r<\frac r2,
\]
because $3073<63^2$. Therefore the entire closed disk
$\overline B(d,R)$ is contained in the region where \eqref{eq31}
holds. That formula proves both level inequalities in \eqref{eq32}.
Finally,
\[
 |c-(t,t)|\le\sqrt2\,|a-t|+56r
 \le(4\sqrt2+56)r,
\]
and every point of the disk is within $r/2$ of $c$. Since
$4\sqrt2+56+1/2<64$, the last assertion follows.
\end{proof}

\begin{proof}[Proof of the level-set assertion in Theorem~\ref{thm5}]
Use the function in \eqref{eq29}, restricted to
$\Omega=(-2,2)^n$. Its analytic properties and monotonicity have
been proved in Proposition~\ref{prop9}. Fix $t\in(-1,1)$, write
$A_t=U^{-1}(t)\cap\Omega$, and let $P_t=(0,t,t,0,\ldots,0)$.
Since $b(t,t)=0$, one has $P_t\in A_t$.

For all sufficiently large $k$, Proposition~\ref{prop10} places
the following circle inside $\Omega$:
\begin{equation}\label{eq33}
 \gamma_{t,k}(e^{i\theta})=
 \bigl(0,d_{t,k}+R_{t,k}(\cos\theta,\sin\theta),0\bigr).
\end{equation}
It lies in $A_t$ and its image is contained in $B(P_t,64r_k)$.
Define on the entire level set
\begin{equation}\label{eq34}
 \rho_{t,k}:A_t\longrightarrow\Sph^1,
 \qquad
 \rho_{t,k}(x,q,\xi)=\frac{q-d_{t,k}}{|q-d_{t,k}|}.
\end{equation}
This is well-defined: if $q=d_{t,k}$, then
\[
 U(x,d_{t,k},\xi)=x^2+w(d_{t,k})>t,
\]
so such a point cannot belong to $A_t$. On the circle in
\eqref{eq33}, the composition $\rho_{t,k}\circ\gamma_{t,k}$
is exactly $(\cos\theta,\sin\theta)$ and has degree one.
Lemma~\ref{lem11} now excludes a manifold chart at $P_t$.

The same map \eqref{eq34} works in every dimension $n\ge3$;
extra coordinates cannot make these circles contractible. This
argument does not rely on an assertion that arbitrary products
preserve nonmanifold points. Finally, a Lipschitz function on the
bounded cube belongs to $W^{1,p}(\Omega)$ for every
$1\le p\le\infty$.
\end{proof}

\begin{remark}\label{rem8}
The proof excludes contraction even by a homotopy that crosses
between the two sets $\{x>0\}$ and $\{x<0\}$ or passes through
arbitrarily many other circles. Such a homotopy would still be
sent by \eqref{eq34} to a contraction of a degree-one circle.
No retraction of a local double, computation of genus, or
description of the entire set $\{w\le t\}$ is needed.
\end{remark}

\begin{remark}\label{rem9}
The failure is topological, not a failure of the Lipschitz coarea
formula. In fact, with the usual normalization of Hausdorff
measure, that formula gives
\[
 \int_{\R}\Hm^{n-1}(U^{-1}(t)\cap\Omega)\dd t
 =\int_\Omega|\nabla U|\dd X<\infty;
\]
see Federer~\cite[Section~3.2]{F1969} or
Evans and Gariepy~\cite[Chapter~3]{EG2015}. Thus almost every
level has finite $(n-1)$-dimensional measure, although every
level in $(-1,1)$ has the nonmanifold point specified above.
\end{remark}

\subsection{Smooth monotone approximation of the example}

\begin{proof}[Proof of the approximation assertion in Theorem~\ref{thm5}]
Truncate scales, rather than individual circles:
\[
 b_N(q)=\sum_{k=1}^N\sum_{j\in\Z}b_{k,j}(q),\qquad
 w_N(y,z)=z+b_N(y,z),\qquad
 U_N(x,y,z,\xi)=x^2+w_N(y,z).
\]
There are only finitely many scales, and the sum over $j$ at
each scale is locally finite. Hence $b_N$, $w_N$, and $U_N$
are smooth on their whole Euclidean spaces.

The proof of Proposition~\ref{prop9} applies inside every
remaining disk. Outside their interiors $\nabla w_N=(0,1)$.
Consequently $w_N$ has no local minimum, and $U_N$ has no
local maximum or minimum. The strictness criterion following Definition~\ref{def1} proves monotonicity.
At any point the omitted sum has at most one nonzero term, so
\begin{equation}\label{eq35}
 \|U-U_N\|_{L^\infty(\Omega)}\le64r_{N+1}.
\end{equation}

Let $E_N\subset(-2,2)^2$ be the union of the omitted open disks,
intersected with this square. A disk at scale $k$ that meets the
square has $a_{k,j}\in[-2-r_k,2+r_k]$. Since the spacing is
$8r_k$ and $r_k\le1/2$, there are at most $2/r_k$ such disks.
Therefore
\begin{equation*}
 |E_N|\le 2\pi\sum_{k>N}r_k
 =4\pi r_{N+1}.
\end{equation*}
Almost everywhere off $E_N$ the gradients of $b$ and $b_N$
agree. To see this directly, exclude the diagonal $\mathcal D$, a
two-dimensional null set, and the countable union of disk
boundaries; local finiteness then applies. On $E_N$ their
gradient difference has magnitude at most $C_b$ from
Proposition~\ref{prop8}. The remaining coordinates contribute
a factor $4^{n-2}$ to volume, whence
\begin{equation}\label{eq36}
 \|\nabla U-\nabla U_N\|_{L^p(\Omega)}^p
 \le C_b^p4^{n-2}\,4\pi r_{N+1}
 \qquad(1\le p<\infty).
\end{equation}
Equations \eqref{eq35} and \eqref{eq36} prove the result.
\end{proof}

\begin{remark}\label{rem10}
The maps in \eqref{eq34} apply to the limiting level set. They
do not imply stability of topology under uniform or strong
Sobolev convergence. Each finite truncation removes the
accumulation of circles responsible for the obstruction at
$P_t$. No claim of strong $W^{1,\infty}$ convergence is made.
\end{remark}

\section{Singular minimizers and higher-dimensional exceptional sets}\label{sec10}

Proposition~\ref{prop1} shows that prescribed Sobolev boundary
values and a nonincreasing energy bound fix every $p$-harmonic
input. We now make the resulting regularity obstruction explicit.
The same homogeneous planar function illustrates both the positive
energy cost of planar smoothing and the failure of a discrete
exceptional set after extension to higher dimensions.

\subsection{An explicit homogeneous 7-harmonic function}

The following is a specialization of the classical separated
variable construction of Kr\'ol and Aronsson
\cite{A1986,K1973}. We include the calculation so that neither
the exponent nor the regularity failure has to be inferred
from a general existence statement.

\begin{lemma}\label{lem12}
There exists a $\pi$-periodic smooth function $f:\R\to\R$
with $f(0)=1$ such that
\[
 h(r\cos\theta,r\sin\theta)=r^{3/2}f(\theta),\qquad h(0)=0,
\]
is $7$-harmonic on $\R^2$, belongs to
$C^{1,1/2}_{\loc}(\R^2)$, is smooth off the origin, and is
not $C^2$ at the origin.
\end{lemma}

\begin{proof}
For $s\ge0$ define
\begin{equation*}
 \Theta(s)=\arctan s-\tfrac12\arctan(3s/2),
 \qquad
 F(s)=\left(\frac{4+9s^2}{4(1+s^2)^3}\right)^{1/4}.
\end{equation*}
Direct differentiation gives
\begin{equation*}
 \Theta'(s)=\frac{1+6s^2}{(1+s^2)(4+9s^2)}>0,
 \qquad
 \frac{F'(s)}{F(s)}=-\tfrac32s\Theta'(s).
\end{equation*}
Moreover, $\Theta(0)=0$ and $\Theta(s)\to\pi/4$ as
$s\to\infty$. Thus $f(\Theta(s))=F(s)$ defines a smooth
function on $[0,\pi/4)$. It satisfies $f(0)=1$, $f'(0)=0$,
and
\[
 f'=-\tfrac32sf,\qquad
 \frac{ds}{d\theta}
 =\frac{(1+s^2)(4+9s^2)}{1+6s^2}.
\]
Substitution, or direct differentiation, now yields
\begin{equation}\label{eq37}
 (3f^2+8(f')^2)f''+23f(f')^2+18f^3=0.
\end{equation}
As $\theta\uparrow\pi/4$,
\begin{equation}\label{eq38}
 f(\theta)\longrightarrow0,
 \qquad
 f'(\theta)\longrightarrow-(3/2)^{3/2}.
\end{equation}
Indeed, $sF(s)\to\sqrt{3/2}$.

The equation \eqref{eq37}, solved for $f''$, is a smooth
ordinary differential equation whenever $(f,f')\ne(0,0)$.
The limits \eqref{eq38} therefore give a smooth continuation
across $\pi/4$. The equation is invariant under reflection
of the independent variable and under changing the sign of
$f$. Uniqueness for the initial value problem implies that
this continuation is obtained by odd reflection about
$\pi/4$:
\[
 f(\theta)=-f(\pi/2-\theta)
 \quad(\pi/4\le\theta\le\pi/2).
\]
Reflect evenly about $0$ and impose
$f(\theta+\pi/2)=-f(\theta)$ thereafter. At the reflection
points the same ordinary differential equation and uniqueness
give smooth matching. This produces a smooth $\pi$-periodic
function. The pair $(f,f')$ never vanishes simultaneously.

Put $G=\tfrac94 f^2+(f')^2$. Since $G>0$, equation
\eqref{eq37} is equivalent to
\begin{equation}\label{eq39}
 (G^{5/2}f')'+6G^{5/2}f=0.
\end{equation}
For a polar function $r^\lambda f(\theta)$, the equation
$\operatorname{div}(|\nabla h|^{p-2}\nabla h)=0$ away
from the origin reduces to
\[
 \left((\lambda^2f^2+(f')^2)^{(p-2)/2}f'\right)'
 +\lambda\bigl(1+(\lambda-1)(p-1)\bigr)
 (\lambda^2f^2+(f')^2)^{(p-2)/2}f=0.
\]
For $p=7$ and $\lambda=3/2$, the coefficient in the second
term is $6$. Thus \eqref{eq39} proves the classical equation
off the origin.

The bounds $|h(q)|\le C|q|^{3/2}$ and
$|\nabla h(q)|\le C|q|^{1/2}$ give differentiability at
the origin with derivative zero, continuity of the gradient,
and local $W^{1,7}$ membership. The gradient has the form
$r^{1/2}A(\theta)$ with $A$ smooth and periodic. To check
its $1/2$-H\"older continuity, take $|q|\ge|q'|$. If
$|q-q'|\ge|q|/2$, use the bound on the gradient itself.
Otherwise the segment from $q$ to $q'$ stays at distance
at least $|q|/2$ from zero, and the homogeneous Hessian
bound $|D^2h|\le C|q|^{-1/2}$ gives
\[
 |\nabla h(q)-\nabla h(q')|
 \le C|q|^{-1/2}|q-q'|
 \le C|q-q'|^{1/2}.
\]
This proves $C^{1,1/2}_{\loc}$ regularity.

To verify the weak equation at zero, integrate the classical
equation on the complement of a disk of radius $\varepsilon$
against a compactly supported smooth test function. The
absolute value of the new boundary term is bounded by
\[
 C\varepsilon\,
 \sup_{|q|=\varepsilon}|\nabla h(q)|^6
 \le C\varepsilon^4\longrightarrow0.
\]
The flux is locally integrable, so passage to the limit proves
\eqref{eq4} on the whole plane.

Finally, $f(0)=f(\pi)=1$ gives $h(s,0)=|s|^{3/2}$.
Its second symmetric difference at zero, divided by $s^2$,
is $2|s|^{-1/2}$, which is unbounded. Hence $h$ is not
$C^2$ at zero.
\end{proof}

\subsection{Product extension and the energy obstruction}

\begin{proof}[Proof of Theorem~\ref{thm6}]
Let $h$ be the function in Lemma~\ref{lem12} and define
\[
 H(q,\zeta)=h(q),\qquad
 (q,\zeta)\in(-1,1)^2\times(-1,1)^{n-2}.
\]
Its $7$-harmonicity in the sense of Definition~\ref{def2}
follows by testing in the $q$
variables for each fixed $\zeta$ and applying Fubini's
theorem. The regularity and the finite $7$-energy on the
cube follow from Lemma~\ref{lem12}. Restriction to a
two-dimensional slice through $(0,0,\zeta)$ shows that
$H$ is not $C^2$ at any point of
\[
 S=\{(0,0)\}\times(-1,1)^{n-2}.
\]
Off $S$ it is smooth.

For completeness, the monotonicity can also be read directly
from the weak equation. Given $V\Subset(-1,1)^n$, put
$m=\max_{\partial V}H$. For $\varepsilon>0$, the function
$(H-m-\varepsilon)_+$ on $V$, extended by zero outside,
has compact support in $V$ and is an admissible Sobolev
test function. Testing shows its gradient has zero
$L^7$ norm, so it vanishes by Poincar\'e's inequality.
Letting $\varepsilon\downarrow0$ gives $H\le m$ on $V$.
Testing the corresponding lower truncation proves the
minimum principle.

By Proposition~\ref{prop1}, any competitor $v$ satisfying the two
variational conditions in the statement equals $H$ almost everywhere.
If a representative of $v$ is smooth on an open neighborhood, it equals the
continuous function $H$ everywhere there: two continuous
functions equal almost everywhere on an open set are
equal everywhere. Thus the representative cannot be
smooth near any point of $S$. If it were smooth near every point outside
a discrete exceptional set, that set would have to
contain $S$. This is impossible when $n\ge3$, since
no point of $S$ is isolated in $S$.
\end{proof}

\begin{remark}\label{rem11}
The obstruction comes from the conjunction of prescribed boundary
values and the one-sided energy bound. Replacing that bound by
convergence of energies does not invoke the rigidity lemma.
There is no contradiction with the planar theorem: when $n=2$,
the exceptional set of the homogeneous function is the single
point $\{0\}$.
Conversely, every approximant supplied by Theorem~\ref{thm1}
for a $p$-harmonic input must equal the input itself. This is
compatible with the local $C^{1,\alpha}$ regularity of
$p$-harmonic functions and shows why that regularity is the
appropriate general conclusion when both variational conditions
are retained.
\end{remark}

\begin{corollary}\label{cor3}
Let $h$ be the function of Lemma~\ref{lem12}, restricted to
$D=(-1,1)^2$. There are smooth strictly monotone Morse functions
$h_j$ with $h_j-h\in W^{1,7}_0(D)$ such that
\begin{gather*}
 \|h_j-h\|_{C^1(D)}+\|h_j-h\|_{W^{1,7}(D)}\longrightarrow0,\\
 E_7(h_j;D)>E_7(h;D),\qquad E_7(h_j;D)\longrightarrow E_7(h;D).
\end{gather*}
All their critical points are saddles, and the changes may be
confined to arbitrarily small neighborhoods of the origin.
\end{corollary}

\begin{proof}
The proof of Lemma~\ref{lem12} shows that
$\tfrac94 f^2+(f')^2>0$. Hence $\nabla h\ne0$ away from the
origin, which is its sole critical point. Apply
Theorem~\ref{thm3} with shrinking error tolerances and
neighborhoods. The strict energy inequality follows from
Proposition~\ref{prop1}, since $h$ is not smooth and thus no
smooth $h_j$ equals it. Strong Sobolev convergence gives
convergence of the energies.
\end{proof}

\section{Variational consequences in a fixed boundary class}\label{sec11}

Let $\Omega\subset\R^n$ be bounded and open, $1<p<\infty$,
and $g\in W^{1,p}(\Omega)$. Define
\begin{equation*}
 \mathcal M_g^p=\{v\in C(\Omega)\cap W^{1,p}(\Omega):
 v\text{ is monotone},\ v-g\in W^{1,p}_0(\Omega)\}.
\end{equation*}
We assume this class is nonempty. For
$\alpha=\alpha(n,p)$ from Theorem~\ref{thm1}, set
\[
 \mathcal M_{g,\alpha}^p=\mathcal M_g^p\cap C^{1,\alpha}_{\loc}(\Omega),
 \qquad
 \mathcal M_{g,\infty}^p=\mathcal M_g^p\cap C^\infty(\Omega).
\]
The affine boundary class is closed, whereas the monotonicity
constraint is not convex in general. No smoothness on
$\overline\Omega$ is imposed on either subclass.

\subsection{Density and continuous integral functionals}

\begin{corollary}\label{cor4}
For every $u\in\mathcal M_g^p$ and every $\eta>0$, there is
$v\in\mathcal M_{g,\alpha}^p$ such that
\[
 \sup_\Omega|v-u|+\|v-u\|_{W^{1,p}(\Omega)}<\eta,
 \qquad \Ep(v;\Omega)\le\Ep(u;\Omega).
\]
If $n=2$, there is also such an approximation by
$v\in\mathcal M_{g,\infty}^p$, with energy bound
$\Ep(v;\Omega)\le\Ep(u;\Omega)+\eta$.
For a non-$p$-harmonic input, strict energy decrease is possible
in both assertions. Consequently
\[
 \inf_{\mathcal M_g^p}\Ep=\inf_{\mathcal M_{g,\alpha}^p}\Ep,
 \qquad
 \inf_{\mathcal M_g^p}\Ep=\inf_{\mathcal M_{g,\infty}^p}\Ep
 \quad\text{if }n=2.
\]
\end{corollary}

\begin{proof}
Use Theorem~\ref{thm1} and, in the plane,
Theorem~\ref{thm4}. In either case
$v-g=(v-u)+(u-g)\in W^{1,p}_0(\Omega)$.
Strong Sobolev convergence gives convergence of the energies.
Approximation of each admissible input and inclusion of the
regular subclasses prove equality of the infima.
\end{proof}

A Carath\'eodory integrand is measurable in its spatial variable
and continuous in its remaining variables for almost every
spatial point~\cite{D2008}. The same approximation argument
applies to any such integrand with a two-sided $p$-growth bound.

\begin{theorem}\label{thm7}
Suppose that $F:\Omega\times\R\times\R^n\to\R$ is a
Carath\'eodory integrand and that
\begin{equation*}
 |F(x,s,\xi)|\le a(x)+C(|s|^p+|\xi|^p),
 \qquad 0\le a\in L^1(\Omega),
\end{equation*}
for almost every $x$ and all $(s,\xi)$. Set
$\mathcal J(v)=\int_\Omega F(x,v,\nabla v)\dd x$.
For every $u\in\mathcal M_g^p$ the approximants from
Theorem~\ref{thm1} satisfy $\mathcal J(u_j)\to\mathcal J(u)$,
and
\begin{equation*}
 \inf_{\mathcal M_g^p}\mathcal J
 =\inf_{\mathcal M_{g,\alpha}^p}\mathcal J.
\end{equation*}
If $n=2$, the smooth approximants from Theorem~\ref{thm4}
also satisfy this convergence, and
\[
 \inf_{\mathcal M_g^p}\mathcal J
 =\inf_{\mathcal M_{g,\infty}^p}\mathcal J.
\]
\end{theorem}

\begin{proof}
Both approximation theorems preserve the boundary class and give
strong $L^p$ convergence of $(u_j,\nabla u_j)$.
Consequently $|u_j|^p+|\nabla u_j|^p$ is uniformly integrable.
Indeed, the inequality
\[
 \bigl||z|^p-|y|^p\bigr|
 \le C_p(|z|+|y|)^{p-1}|z-y|
\]
and H\"older's inequality give $L^1$ convergence of the
corresponding $p$th powers. The growth bound then makes
$F(x,u_j,\nabla u_j)$ uniformly integrable.

Every subsequence has a further subsequence on which both values
and gradients converge almost everywhere. The Carath\'eodory
property gives almost-everywhere convergence of the integrands,
and Vitali's theorem gives convergence in $L^1$ along that
subsequence~\cite{B2011,D2008}. The subsequence criterion
proves convergence for the full sequence. Approximation of every
member of $\mathcal M_g^p$ and inclusion of the regular subclasses
prove the equalities, also when the common infimum is $-\infty$.
\end{proof}

\begin{corollary}\label{cor5}
Under the hypotheses of Theorem~\ref{thm7}, every minimizing
sequence in $\mathcal M_g^p$ can be replaced by a minimizing
sequence in $\mathcal M_{g,\alpha}^p$ whose difference from
the original sequence tends to zero uniformly and in $W^{1,p}$.
If $n=2$, the replacement sequence can belong to
$\mathcal M_{g,\infty}^p$.
\end{corollary}

\begin{proof}
Approximate the $j$th member so that the sum of the uniform,
Sobolev, and absolute functional errors is less than $1/j$.
Theorem~\ref{thm7} permits this choice independently for
every $j$.
\end{proof}

Thus there is no Lavrentiev gap between the specified monotone
classes for the indicated functionals. This does not assert
existence of a smooth minimizer. For example, let $h$ be the
homogeneous $7$-harmonic function of Lemma~\ref{lem12} on
$D=(-1,1)^2$ and use it as the boundary datum. Then
\begin{equation*}
 \inf_{\mathcal M_{h,\infty}^{7}}E_7
 =\min_{\mathcal M_h^{7}}E_7=E_7(h;D),
\end{equation*}
but the smooth infimum is not attained, by
Proposition~\ref{prop1}. Corollary~\ref{cor3} gives
explicitly the type of smooth recovery sequence needed here.

The growth assumptions cover bounded measurable weights multiplying
$|\nabla v|^p$ and continuous lower-order terms of the stated
growth. They exclude singular Jacobian dependence, higher growth,
and additional constraints not stable under strong Sobolev
convergence. Lavrentiev phenomena in more general variational
problems, such as those of Ball and Mizel~\cite{BM1985}, are
not ruled out by this result.

\subsection{Maps with monotone coordinates}

\begin{corollary}\label{cor6}
Suppose $F=(f^1,\ldots,f^m)\in C(\Omega;\mathbb R^m)
\cap W^{1,p}(\Omega;\mathbb R^m)$ and each $f^i$ is monotone.
There are maps $F_j=(f_j^1,\ldots,f_j^m)$ of class
$C^{1,\alpha}_{\loc}$ with monotone coordinates such that
\[
 F_j\to F\quad\text{uniformly and strongly in }W^{1,p},
 \qquad F_j-F\in W^{1,p}_0(\Omega;\mathbb R^m),
\]
and
\[
 \int_\Omega|\nabla f_j^i|^p\dd x
 \le\int_\Omega|\nabla f^i|^p\dd x
 \qquad(1\le i\le m).
\]
For $p=2$, $\int_\Omega|DF_j|^2\le\int_\Omega|DF|^2$, where
$|DF|$ is the Frobenius norm.
If $n=2$, there is also a sequence of smooth maps with monotone
coordinates and the same convergence and boundary properties;
for this sequence each coordinate energy converges to that of
the original coordinate. A nonincreasing bound holds coordinate
by coordinate whenever each $p$-harmonic coordinate is already
smooth.
\end{corollary}

\begin{proof}
Apply Theorem~\ref{thm1} to the finitely many coordinates and
choose the same sequence of error tolerances. Summing the squared
coordinate gradients gives the Frobenius-energy inequality.
For $n=2$, apply Theorem~\ref{thm4} to each coordinate instead.
Its strict assertion treats every non-$p$-harmonic coordinate;
leave an already smooth $p$-harmonic coordinate unchanged.
This gives the final assertion.
\end{proof}

Every coordinate of a homeomorphism between Euclidean domains
is Lebesgue monotone: the image of an interior neighborhood is
open, and a coordinate projection takes both larger and smaller
values there. Thus Corollary~\ref{cor6} applies to Sobolev
homeomorphisms. It retains their scalar maximum and minimum
principles but does not retain injectivity, orientation, or a
prescribed image. This is the precise relation with the
homeomorphic approximation problem discussed in~\cite{N2020a}
and with the mapping results~\cite{CDV2026,HP2018,IKO2011,IKO2012,IO2016}.

\section{Coarea, perimeter convergence, and level-set topology}\label{sec12}

In this section $\Omega\subset\R^n$ is bounded and open and
$1<p<\infty$. The arguments apply to the sequences in
Theorems~\ref{thm1} and~\ref{thm4}, and also to the explicit
sequence in Theorem~\ref{thm5}. In fact, the analytic conclusions
only need uniform and strong $W^{1,p}$ convergence.

For a measurable set $E\subset\Omega$, its relative perimeter is
\[
 P(E;\Omega)=|D\chi_E|(\Omega),
\]
where $D\chi_E$ denotes the distributional gradient and $|D\chi_E|$
its total variation measure. A function belongs to $BV(\Omega)$
when it is integrable and its distributional gradient is a finite
vector-valued Radon measure. Strict convergence in $BV$ means
convergence in $L^1$ together with convergence of the total
variations. We use these conventions and the coarea theorem
from~\cite{AFP2000}. In particular, for $v\in W^{1,1}(\Omega)$,
\begin{equation*}
 \int_{\mathbb R}P(\{v>t\};\Omega)\dd t
 =\int_\Omega|\nabla v|\dd x.
\end{equation*}

\begin{corollary}\label{cor7}
Let $u_j,u\in C(\Omega)\cap W^{1,p}(\Omega)$ with $u_j\to u$
uniformly and strongly in $W^{1,p}$, and put
$E_t=\{u>t\}$ and $E_{j,t}=\{u_j>t\}$. Then
\begin{align}
 &\int_{\mathbb R}
   \|\chi_{E_{j,t}}-\chi_{E_t}\|_{L^1(\Omega)}\dd t
   =\|u_j-u\|_{L^1(\Omega)}\longrightarrow0,\label{eq40}\\
 &\int_{\mathbb R}
   |P(E_{j,t};\Omega)-P(E_t;\Omega)|\dd t
   \longrightarrow0.\label{eq41}
\end{align}
There is a subsequence along which
$\chi_{E_{j,t}}\to\chi_{E_t}$ strictly in $BV(\Omega)$ for
almost every $t$. These assertions hold with $\Omega$ replaced by
any fixed open subset of $\Omega$.
\end{corollary}

\begin{proof}
For real numbers $a,b$,
\[
 \int_{\mathbb R}
 |\chi_{\{a>t\}}-\chi_{\{b>t\}}|\dd t=|a-b|.
\]
Tonelli's theorem proves~\eqref{eq40}. Since $|\Omega|<\infty$,
strong $W^{1,p}$ convergence implies strong $W^{1,1}$ convergence.
Writing $a_j(t)=P(E_{j,t};\Omega)$ and $a(t)=P(E_t;\Omega)$,
the coarea formula gives
\begin{equation}\label{eq42}
 \int_{\mathbb R}a_j(t)\dd t
 \longrightarrow\int_{\mathbb R}a(t)\dd t<\infty.
\end{equation}

Take any subsequence. By~\eqref{eq40}, a further subsequence has
$\chi_{E_{j,t}}\to\chi_{E_t}$ in $L^1(\Omega)$ for almost every
$t$. Lower semicontinuity of perimeter gives
$a(t)\le\liminf_j a_j(t)$ at such levels. Consequently
$(a-a_j)_+\to0$ almost everywhere and
$0\le(a-a_j)_+\le a\in L^1(\mathbb R)$. Dominated convergence,
\eqref{eq42}, and the identity
\[
 \int|a_j-a|=\int a_j-\int a+2\int(a-a_j)_+
\]
give $\|a_j-a\|_{L^1(\mathbb R)}\to0$ on the further subsequence.
The subsequence criterion proves~\eqref{eq41} for the full sequence.
Finally choose a subsequence for which the sum of the integrals
in~\eqref{eq40} and~\eqref{eq41} is summable. Both integrands then
converge to zero for almost every $t$, which is precisely the
asserted strict $BV$ convergence. The proof on a fixed open subset
is identical.
\end{proof}

Perimeter is measured on the reduced boundary and does not control
the topology of the entire level set. Corollary~\ref{cor7}
therefore remains meaningful for the example in
Theorem~\ref{thm5}. It does not assert convergence of the
distributional gradients in total variation norm, which is stronger
than strict $BV$ convergence.

\begin{corollary}\label{cor8}
For each $n\ge3$, there are smooth monotone $U_N$ and a Lipschitz
monotone $U$ on $\Omega=(-2,2)^n$, converging uniformly and strongly in
every $W^{1,p}$ with $1\le p<\infty$, such that for almost every
$t\in(-1,1)$ all of the sets $U_N^{-1}(t)\cap\Omega$ are smooth embedded
hypersurfaces whereas $U^{-1}(t)\cap\Omega$ is not a topological manifold.
After taking a subsequence, their superlevel characteristic
functions also converge strictly in $BV$ for almost every such $t$.
\end{corollary}

\begin{proof}
Use the explicit sequence in Theorem~\ref{thm5}. Sard's
theorem~\cite{S1942} gives a null set of critical values for each
$U_N$. Outside the countable union of those null sets, the regular
value theorem~\cite{M1965} makes every $U_N^{-1}(t)\cap\Omega$ a smooth
embedded hypersurface. The limiting level is nonmanifold for
every $t\in(-1,1)$ by Theorem~\ref{thm5}.
Apply Corollary~\ref{cor7} to this same sequence and intersect
the two full-measure sets of levels to obtain the strict $BV$
assertion. Thus even that stronger measure-theoretic convergence
does not force a manifold structure on the limiting level.
\end{proof}

\begin{corollary}\label{cor9}
Under the hypotheses of Corollary~\ref{cor7}, suppose also
that all $u_j$ are smooth. Outside a single null set of levels,
each $t$ is a regular value of every $u_j$, and
\[
 P(\{u_j>t\};\Omega)=\Hm^{n-1}(\{u_j=t\}\cap\Omega).
\]
These identities hold simultaneously with the almost-everywhere
strict $BV$ convergence along the subsequence of that corollary.
\end{corollary}

\begin{proof}
Remove the countable union of the null sets of critical values
given by Sard's theorem~\cite{S1942}. At every remaining value
the regular value theorem gives a smooth hypersurface, and the
superlevel set lies locally on one side of it. The perimeter
identity follows from the smooth-boundary formula, using
exhaustion if the level set is noncompact. Intersect with the
full-measure set supplied by Corollary~\ref{cor7}.
\end{proof}

There is also a useful weighted form in which one integrates over
the full level sets.  The scalar Sobolev coarea formula goes back
to Federer \cite{F1969}; a convenient formulation within the theory
of Sobolev mappings is \cite[Theorem~1.1]{MSZ2003} of Mal\'y,
Swanson and Ziemer. For scalar functions ($m=1$), that theorem
permits every $p\ge1$, including the endpoint $p=1$.
It applies here to the continuous representative, which agrees
at every point with the representative obtained from local averages.
The weighted version follows by first using simple measurable
weights and then the usual approximation argument.

\begin{proposition}\label{prop11}
Under the hypotheses of Corollary~\ref{cor7}, for every $\phi\in C_c(\Omega\times\R)$,
\begin{equation}\label{eq43}
 \begin{split}
 &\int_{\R}\int_{\{u_j=t\}}\phi(x,t)\dd\Hm^{n-1}(x)\dd t\\
 &\hspace{2em}\longrightarrow
 \int_{\R}\int_{\{u=t\}}\phi(x,t)\dd\Hm^{n-1}(x)\dd t.
 \end{split}
\end{equation}
\end{proposition}

\begin{proof}
Coarea identifies the two expressions with
\[
 \int_\Omega\phi(x,u_j(x))|\nabla u_j(x)|\dd x
 \quad\text{and}\quad
 \int_\Omega\phi(x,u(x))|\nabla u(x)|\dd x,
\]
respectively.
On the compact projection of $\supp\phi$ to $\Omega$, the
functions $u_j$ converge uniformly to $u$, so
$\phi(x,u_j(x))\to\phi(x,u(x))$ uniformly there.  Also
$|\nabla u_j|\to|\nabla u|$ in $L^1(\Omega)$.  Splitting the
difference of these integrals into the difference of the weights and
the difference of the gradient norms proves \eqref{eq43}.
\end{proof}

\begin{remark}\label{rem12}
When smooth monotone approximants are available, their regular
level sets are smooth hypersurfaces on which the averaged coarea
integrals can be evaluated. Along a subsequence, the associated
superlevel sets converge strictly in $BV$ at almost every level.
These statements do not assert convergence of curvature or
preservation of the topology of individual level sets.  The
metric coarea inequality of Esmayli, Ikonen and Rajala \cite{EIR2023}
and the subsequent result of Meier and Ntalampekos \cite{MN2024}
address a different issue, namely the validity of coarea estimates
when the ambient surface has only a metric structure.  Neither
supplies the local smooth replacement required in \cite[Question~1.7]{N2020a}.
\end{remark}

\section{Nonlinear fluxes and approximate field equations}\label{sec13}

Let $D\subset\R^n$ be bounded and open. For $1<p<\infty$ define
\begin{equation*}
 A_p(\xi)=|\xi|^{p-2}\xi,
 \qquad V_p(\xi)=|\xi|^{(p-2)/2}\xi,
 \qquad p'=\frac{p}{p-1},
\end{equation*}
with both maps set equal to zero at the origin.  The first is the
constitutive flux of the $p$-Laplace equation.  The second is a
standard change of gradient variable in its nonlinear error analysis;
see \cite{BL1993,DK2008}.  Notice the identity
$|V_p(\xi)|^2=|\xi|^p$.

For a bounded open set $D$, we use the equivalent negative Sobolev
norm
\begin{equation}\label{eq44}
 \|T\|_{W^{-1,p'}(D)}
  =\sup\{|\langle T,\varphi\rangle|:
       \varphi\in W^{1,p}_0(D),\ \|\nabla\varphi\|_{L^p(D)}\le1\}.
\end{equation}
Thus $W^{-1,p'}(D)$ is the dual of $W^{1,p}_0(D)$, equipped here
with the gradient norm; see \cite{B2011}.  If a smooth function has
critical points and $p<2$, its $p$-Laplacian below is understood in
the distributional sense.

\begin{proposition}\label{prop12}
Suppose that $w_j\to u$ strongly in $W^{1,p}(D)$.  Then
\begin{equation}\label{eq45}
 A_p(\nabla w_j)\longrightarrow A_p(\nabla u)
     \quad\text{in }L^{p'}(D,\R^n),
 \qquad
 V_p(\nabla w_j)\longrightarrow V_p(\nabla u)
     \quad\text{in }L^2(D,\R^n).
\end{equation}
More precisely, with all norms taken on $D$,
\begin{equation}\label{eq46}
 \|A_p(\nabla w)-A_p(\nabla u)\|_{p'}
 \le
 \begin{cases}
 C_p(\|\nabla w\|_p+\|\nabla u\|_p)^{p-2}
           \|\nabla w-\nabla u\|_p,&p\ge2,\\
 C_p\|\nabla w-\nabla u\|_p^{p-1},&1<p<2.
 \end{cases}
\end{equation}
In general,
\[
 \|\Delta_pw_j-\Delta_pu\|_{W^{-1,p'}(D)}
  \le\|A_p(\nabla w_j)-A_p(\nabla u)\|_{p'}\longrightarrow0.
\]
If $u$ is $p$-harmonic in $D$, then
\begin{equation}\label{eq47}
 \|\Delta_pw_j\|_{W^{-1,p'}(D)}
  \le\|A_p(\nabla w_j)-A_p(\nabla u)\|_{p'}\longrightarrow0.
\end{equation}
\end{proposition}

\begin{proof}
The elementary vector inequalities behind \eqref{eq46} are
\[
 |A_p(\xi)-A_p(\zeta)|\le
 \begin{cases}
 C_p(|\xi|+|\zeta|)^{p-2}|\xi-\zeta|,&p\ge2,\\
 C_p|\xi-\zeta|^{p-1},&1<p<2.
 \end{cases}
\]
For $p\ge2$ they follow by integrating the derivative of $A_p$
on the line segment.  For $p<2$, put $d=|\xi-\zeta|$ and
$R_0=\max\{|\xi|,|\zeta|\}$.  If $d\ge R_0/2$, the triangle
inequality gives the claimed bound.  If $d<R_0/2$, the segment stays
at distance at least $R_0/2$ from zero.  The derivative bound gives
$C_pR_0^{p-2}d\le C_pd^{p-1}$. The case $d=0$ is immediate.
H\"older's inequality proves \eqref{eq46}; for $p>2$ the exponents
are $p/(p-2)$ and $p$, and for $p=2$ the estimate is the identity
for the linear flux.

For the assertion about $V_p$, every subsequence has a further
subsequence on which the gradients converge almost everywhere.
Continuity of $V_p$ gives pointwise convergence there, and
\[
 |V_p(\nabla w_j)-V_p(\nabla u)|^2
 \le 2|\nabla w_j|^p+2|\nabla u|^p.
\]
The right-hand side is uniformly integrable by strong $L^p$
convergence.  Vitali's theorem and the subsequence criterion prove
the $L^2$ convergence in \eqref{eq45}.

Testing the difference of the fluxes against $\nabla\varphi$
first proves the stated bound for $\Delta_pw_j-\Delta_pu$.
When $\Delta_pu=0$, the same computation gives
\[
 |\langle\Delta_pw_j,\varphi\rangle|
 =\left|\int_D\bigl(A_p(\nabla w_j)-A_p(\nabla u)\bigr)
                    \cdot\nabla\varphi\dd x\right|.
\]
H\"older's inequality and \eqref{eq44} imply \eqref{eq47}.
\end{proof}

\begin{corollary}\label{cor10}
Let $u\in W^{1,p}(D)$ be nonconstant and planar $p$-harmonic, where
$D$ is bounded and connected.  There are smooth strictly monotone
Morse functions $w_j$ with $w_j-u\in W^{1,p}_0(D)$, converging to
$u$ uniformly and strongly in $W^{1,p}$, whose fluxes and transformed
gradients satisfy \eqref{eq45}, and whose equation residuals satisfy
\eqref{eq47}.  Every critical point of $w_j$ is a saddle.
If the changes are confined to a compact set $K\subset D$, then
$\supp(\Delta_pw_j)\subset K$.
\end{corollary}

\begin{proof}
Apply the second assertion of Corollary~\ref{cor2} and then
Proposition~\ref{prop12}.  The resulting functions have only saddle
critical points and hence have no local extrema.  If $w_j=u$
outside $K$, the two fluxes agree there almost everywhere.  Testing
against functions supported in $D\setminus K$ proves the support
assertion.
\end{proof}

This has a direct interpretation for a scalar potential in a fixed
planar region with a power-law flux.  With a constant $\kappa>0$,
write
\begin{equation*}
 \mathbf J=-\kappa A_p(\nabla u),\qquad \diver\mathbf J=0.
\end{equation*}
The approximation provides a smooth potential $w_j$ in the same
boundary class, without introducing interior maxima or minima, and
$\mathbf J_j=-\kappa A_p(\nabla w_j)$ satisfies
\[
 \|\mathbf J_j-\mathbf J\|_{L^{p'}(D)}\to0,
 \qquad \|\diver\mathbf J_j\|_{W^{-1,p'}(D)}\to0.
\]
Power-law Hele--Shaw flow gives one source of this elliptic
constitutive law; see Aronsson and Janfalk \cite{AJ1992}.  The
assertion concerns the elliptic potential on a fixed source-free
region.  A time-dependent free boundary requires separate estimates.

The same convergence is compatible with the nonlinear error variables
used in finite element analysis \cite{BL1993,DK2008}.  It permits a
smooth monotone potential to be used as an intermediate approximation
while controlling its flux and weak equation error.  The argument does
not by itself select a mesh, prove a convergence rate, or show that
interpolation of the smooth function preserves monotonicity.

\section{Morse approximation and persistence of sublevel sets}\label{sec14}

We give a further planar topological application in which the uniform error
has an independent role.  Fix a compact triangulable space $X$ and
a coefficient field $\mathbb F$.  For a continuous function
$f:X\to\R$, its sublevel-set filtration is
\begin{equation*}
 X_t^f=\{x\in X:f(x)\le t\},\qquad t\in\R.
\end{equation*}
Inclusions for $s\le t$ induce maps on singular homology
$H_k(X_s^f;\mathbb F)\to H_k(X_t^f;\mathbb F)$.
The resulting family is the degree-$k$ persistence module.  We use
the tame setting of \cite{CEH2007}: these homology groups are finite
dimensional and only finitely many parameter values change the
persistent homology.  Its persistence diagram $\Dgm_k(f)$ records
the birth and death parameters of the interval summands, with
multiplicities.  An essential interval has death parameter $+\infty$.
Definitions and the interval interpretation are discussed in
\cite{CEH2007,EH2010}.

The bottleneck distance $d_B$ is the infimum, over matchings of two
diagrams after adjoining the diagonal with unlimited multiplicity,
of the largest $\ell^\infty$ distance between matched points.
The distance of a finite point $(b,d)$ to the diagonal is
$(d-b)/2$.  Essential intervals are matched to essential intervals
using the distance between their birth parameters.  The stability
theorem of Cohen-Steiner, Edelsbrunner and Harer states that
\begin{equation}\label{eq48}
 d_B(\Dgm_k(f),\Dgm_k(g))\le\|f-g\|_{L^\infty(X)}
\end{equation}
for continuous tame functions in this setting \cite{CEH2007}.

\begin{proposition}\label{prop13}
Assume the hypotheses of Theorem~\ref{thm3}, and let
$X\subset V$ be compact and triangulable.  For every $\eta>0$
there is a smooth strictly monotone replacement $w$ satisfying all
the conclusions of that theorem and
\begin{equation}\label{eq49}
 \|w-u\|_{W^{1,p}(V)}<\eta,
 \qquad \|w-u\|_{L^\infty(X)}<\eta.
\end{equation}
For every degree $k$ in which $u|_X$ and $w|_X$ are tame,
\begin{equation*}
 d_B(\Dgm_k(u|_X),\Dgm_k(w|_X))<\eta.
\end{equation*}
In particular, a finite persistence interval of length greater than
$2\eta$ cannot be removed by matching it to the diagonal within
this error bound.
\end{proposition}

\begin{proof}
Choose $w$ from Theorem~\ref{thm3} with error smaller than $\eta$.
This gives \eqref{eq49}.  If
$\delta=\|w-u\|_{L^\infty(X)}$, then for every $t$
\[
 X_t^u\subset X_{t+\delta}^w,
 \qquad X_t^w\subset X_{t+\delta}^u.
\]
These inclusions are the two shifted comparisons underlying
persistence stability.  Applying \eqref{eq48} gives
$d_B\le\delta<\eta$.  A matching with cost less than $\eta$
exists by the definition of the infimum.  A point whose distance to
the diagonal is greater than $\eta$ cannot be matched to it, proving
the last assertion.
\end{proof}

\begin{remark}\label{rem13}
The new geometric input in Proposition~\ref{prop13} is the ability
to replace an isolated $p$-harmonic critical point by precisely the
number of Morse saddles dictated by its index, while controlling
both the uniform and Sobolev errors and leaving the exterior fixed.
The diagram estimate is an application of the established stability
theorem.  Tameness of restrictions to an arbitrary compact $X$ does
not follow from interior Morse nondegeneracy and is an explicit
hypothesis.  Likewise, preservation of long persistence intervals
does not imply equality of every Betti number at every level.
\end{remark}

\section*{Acknowledgments}
This work was supported by the Visiting Scholar Program of Chern Institute of Mathematics at Nankai University. The author would
like to express his hearty thanks to Chern Institute of Mathematics provided very comfortable research environments to him worked as visiting scholar.
This work was also supported by the Guangdong Basic and Applied Basic Research Foundation (Nos.\ 2022A1515110967 and 2023A1515011809).  AI-assisted tools were used during the preparation and revision of the manuscript for mathematical checking, literature comparison, editorial organization, and typesetting. The author takes full responsibility for the mathematical statements, derivations, citations, and conclusions in the final manuscript.

\end{document}